\documentclass{amsproc}

\usepackage{amsmath}
\usepackage{amssymb}
\usepackage[mathcal]{eucal}
\usepackage[bb=pazo]{mathalfa}
\usepackage{rotating}
\usepackage{graphicx}
\usepackage[tableposition=below]{caption}
\usepackage{pigpen}
\usepackage{multirow}
\usepackage{mathtools}
\usepackage{array}
\newcolumntype{H}{>{\setbox0=\hbox\bgroup}c<{\egroup}@{}}
\usepackage[textsize=tiny]{todonotes}

\DeclareGraphicsExtensions{.png,.pdf,.eps}
\usepackage[all,2cell,cmtip]{xy}
\usepackage{tikz}
\usepackage{color}
\usepackage{bm}
\usepackage{longtable}
\usepackage{soul}
\newtheorem{theorem}{Theorem}[section]
\newtheorem{lemma}[theorem]{Lemma}
\newtheorem{corollary}[theorem]{Corollary}
\newtheorem{proposition}[theorem]{Proposition}

\theoremstyle{definition}

\newtheorem{definition}[theorem]{Definition}

\newtheorem{remark}[theorem]{Remark}

\newtheorem{example}[theorem]{Example}

\def\C{{\mathbb C}}

\def\G{{\mathbb G}}

\def\O{{\mathbb O}}
\def\P{{\mathbb P}}
\def\Q{{\mathbb Q}}
\def\R{{\mathbb R}}
\def\Z{{\mathbb Z}}

\def\cC{{\mathcal C}}
\def\cD{{\mathcal D}}
\def\cE{{\mathcal E}}
\def\cF{{\mathcal F}}

\def\cK{{\mathcal K}}
\def\cL{{\mathcal L}}

\def\cN{{\mathcal{N}}}
\def\cO{{\mathcal{O}}}

\def\cQ{{\mathcal{Q}}}

\def\cS{{\mathcal S}}

\def\Q{{\mathbb{Q}}}
\def\G{{\mathbb{G}}}

\def\fg{{\mathfrak g}}
\def\fh{{\mathfrak h}}

\def\PGL{\operatorname{PGl}}
\def\SL{\operatorname{Sl}}

\def\PSO{\operatorname{PSO}}
\def\SO{\operatorname{SO}}
\def\SP{\operatorname{Sp}}
\def\PSP{\operatorname{PSp}}
\def\Spin{\operatorname{Spin}}

\def\Coeff{\operatorname{Coeff}}
\def\Nest{\operatorname{Nest}}
\def\HH{\operatorname{H\hspace{0.5pt}}}

\def\Mo{\operatorname{\hspace{0cm}M}}
\def\Na{\operatorname{\hspace{0cm}L}}

\def\DA{{\rm A}}
\def\DB{{\rm B}}
\def\DC{{\rm C}}
\def\DD{{\rm D}}
\def\DE{{\rm E}}
\def\DF{{\rm F}}
\def\DG{{\rm G}}

\def\lra{\longrightarrow}

\def\ra{\rightarrow}
\def\lra{\longrightarrow}

\def\operatorname#1{\mathop{\rm #1}\nolimits}

\def\Proj{\operatorname{Proj}}

\def\Hom{\operatorname{Hom}}

\def\Pic{\operatorname{Pic}}

\def\id{\operatorname{id}}

\def\rk{\operatorname{rk}}

\def\deg{\operatorname{deg}}

\def\det{\operatorname{det}}

\def\NU{{\operatorname{N^1}}}
\def\HH{\operatorname{H\hspace{0.5pt}}}

\def\ad{\operatorname{ad}}

\def\Cox{{\operatorname{\hspace{0cm}h}}}

\newcommand{\ol}[1]{\overline{#1}}
\newcommand{\pb}{\ar@{}[dr]|(.50){\text{\pigpenfont J}}}

\makeatletter
 
\newcommand*\wthelper[2]{%
        \hbox{\dimen@\accentfontxheight#1%
                \accentfontxheight#11.15\dimen@
                $\m@th#1\widetilde{#2}$%
                \accentfontxheight#1\dimen@
        }%
}

\newcommand*\accentfontxheight[1]{%
        \fontdimen5\ifx#1\displaystyle
                \textfont
        \else\ifx#1\textstyle
                \textfont
        \else\ifx#1\scriptstyle
                \scriptfont
        \else
                \scriptscriptfont
        \fi\fi\fi3
}
\makeatother

\newcommand{\shse}[3]{0 ~\ra ~#1~ \lra ~#2~ \lra ~#3~ \ra~ 0}

\makeindex

\begin{document}

\title[Nestings of rational homogeneous varieties]{Nestings of rational homogeneous varieties}

\author[R. Mu\~noz]{Roberto Mu\~noz}
\address{Departamento de Matem\'atica Aplicada, ESCET, Universidad
Rey Juan Carlos, 28933-M\'ostoles, Madrid, Spain}
\email{roberto.munoz@urjc.es}
\author[Occhetta]{Gianluca Occhetta}
\address{Dipartimento di Matematica, Universit\`a di Trento, via
Sommarive 14 I-38123 Povo di Trento (TN), Italy} 
\email{gianluca.occhetta@unitn.it}
\author[Sol\'a Conde]{Luis E. Sol\'a Conde}
\address{Dipartimento di Matematica, Universit\`a di Trento, via
Sommarive 14 I-38123 Povo di Trento (TN), Italy}
\email{eduardo.solaconde@unitn.it}

\begin{abstract}
In this paper we study the existence of sections of universal bundles on rational homogeneous varieties -- called nestings -- classifying them completely on rational homogeneous varieties $G/P$ in the case where $G$ is a simple group of classical type and $P$ is a parabolic subgroup of $G$. In particular we show that, under this hypothesis, nestings do not exist unless there exists a proper algebraic subgroup of the automorphism group acting transitively on the base variety.
\end{abstract}

\maketitle

\section{Introduction}\label{sec:intro2}

Let $\G(k,n)$ be the Grassmannian of linear spaces of dimension $k$ in the projective space of dimension $n$.  In \cite{DCR}, De Concini and Reichstein studied the existence of morphisms $f: \G(k,n)\to \G(r,n)$,  $k<r$, mapping a $\P^k\subset\P^n$ to a $\P^r\subset\P^n$ containing it. They called such morphisms {\em nesting maps} and showed  that algebraic nestings only exist in the case $n$ odd and $\{k,r\}=\{0,n-1\}$. Moreover, such maps are described via the choice of an alternating form on the corresponding $(n+1)$-dimensional ambient vector space $V$: a one dimensional vector subspace of $V$ is mapped to its orthogonal complement with respect to the chosen form. 

The problem of describing nesting maps may be extended to different contexts; for instance, one could ask for the existence of nesting maps on prescribed subsets of $\G(k,n)$, as the sets parametrizing subspaces that are isotropic (or non isotropic) with respect to a given symmetric or skew-symmetric form in $V$.  Let us observe that to give a nesting map $f: \G(k,n)\to \G(r,n)$ is equivalent to giving a section of the first projection from the  flag variety $F_{k,r}=\{(\P^k,\P^r)\,|\, \P^k \subset \P^r\} \subset \G(k,n) \times  \G(r,n)$ to $\G(k,n)$. In other words, nesting maps are algebraic ways of completing $\P^k\in \G(k,n)$ to a flag $(\P^k\subset\P^r)$ in $F_{k,r}$. Seen in this way, the notion of nesting map makes perfect sense also in the case $k>r$ and, more generally, one could pose the question of the existence of sections of the natural projections between other flag varieties  of $\P^n$.
 
On the other hand, the problem can be extended to other types of varieties; 
Grassmannians and flag varieties are  examples of rational homogeneous varieties, that is, projective algebraic varieties on which a semisimple algebraic group $G$ acts transitively.  In the case of Grassmannians and flag varieties, $G$ is $\PGL(n+1)$ and there exist parabolic subgroups $P, P' \subset G$ such that $P \subset P' \subset G$,  $F_{k,r}=G/P$ and $\G(k,n)=G/P'$;  then a nesting map is nothing  but a section of the natural projection $G/P \to G/P'$. In this context, the concept of nesting map can be naturally generalized as follows: given $G$ a semisimple algebraic group, and $P \subset P' \subset G$ two parabolic subgroups, we define a {\it nesting} for $P \subset P' \subset G$ as a section of the natural projection $G/P \to G/P'$. The term nesting is justified in this setting by the fact that rational homogeneous varieties can always be seen as subvarieties of flag varieties of a suitable projective space, so a nesting can always be seen as a way of completing flags, as in the case considered by De Concini and Reichstein. 
 
 The goal of this paper is to give a complete description of the possible nestings when $G$ is simple of classical type (over the field of complex numbers). We will show that nestings are rather exceptional: besides the cases described in  \cite{DCR}, there is another infinite family,  and a sporadic case.
 
To describe the infinite family, let $Q^{2n-2}$ be an even dimensional quadric, and denote by $S_a$ and $S_b$ the two Spinor varieties, parametrizing the two families  of linear subspaces in $Q^{2n-2}$ of maximal dimension $(n-1)$.
The linear spaces of dimension $n-2$ contained in $Q^{2n-2}$ are parametrized by the incidence variety $\{(\P^{n-1}_a,\P^{n-1}_b)\in S_a \times S_b\,|\,\dim(\P^{n-1}_a \cap \P^{n-1}_b)=n-2\}$ and we have nestings parametrized by smooth hyperplane sections $H$ of $Q^{2n-2}$, sending $\P^{n-1}_a$ (resp. $\P^{n-1}_b$) to the flag $(\P^{n-1}_a \cap H \subset \P^{n-1}_a)$ (resp. $(\P^{n-1}_b \cap H \subset \P^{n-1}_b)$).

Finally, in the sporadic case, nestings are maps from a five dimensional quadric $Q^5$ to the flag of points and planes in $Q^5$, sending a point $p$ to a flag $(p, \pi)$, with $p \in \pi \subset Q^5$; these maps admit a geometric description in terms of complexified octonions (see Section \ref{ssec:exB3}).

To state the main result of the paper, let us recall first that, given a semisimple group $G$, a rational homogeneous variety $G/P'$ can be described by marking a subset $I$ of the set $D$ of nodes of the Dynkin diagram $\cD$ of $G$; we will thus denote the variety $G/P'$  by $\cD(I)$. Moreover, the parabolic subgroups $P \subset G$ contained in $P'$ correspond to the subsets of $D$ containing $I$ (see Section \ref{ssec:prelim1} for details).
In these terms, given a Dynkin diagram $\cD$, and two disjoint nonempty subsets $I,J$ of nodes of $\cD$, a {\em nesting of type $(\cD,I,J)$}, is a section of the projection $\cD(I\cup J)\to \cD(I)$, and the main theorem can be stated as follows:

\begin{theorem}\label{thm:main}
Let $G$ be a simple algebraic group whose associated Dynkin diagram $\cD$ is of classical type, and let $I,J$ be two disjoint nonempty sets of nodes of $\cD$ such that $(\cD,I,J)$ admits a nesting. Then  
$(\cD,I,J)$ is isomorphic to one of the following:
$$
(\DA_{2n-1},1,2n-1)\,\,\,\, \mbox{$n\geq 2$},\quad (\DB_3,1,3),\quad (\DD_n,n-1,n)\,\,\,\,n\geq 4.
$$
\end{theorem}

Furthermore, we show that nestings of type $(\cD,I,J)$ are tightly related to the existence of a smaller group acting transitively on  $\cD(I)$.  In fact,
in the case of $(\DA_{2n-1},1,2n-1)$, nestings are determined by 
the choice of a structure of $\PSP(2n)$-variety on $\P^{2n-1}$ (see \cite{DCR}). We prove that a similar property holds in the other two cases, namely, that  nestings of types $(\DB_3,1,3)$ (resp. $(\DD_n,n-1,n)$) are determined by the choice of a structure of $\DG_2$-variety of $\DB_3(1)$, (resp. a structure of $\PSO_{2n-1}$-variety of $\DD_n(n-1)$).
In this way we get a full description of the parameter spaces of nestings, $\Nest(\cD,I,J)$, in all cases. 

The fact that, for Dynkin diagrams  of classical type, a nesting of type $(\cD,I,J)$ exists if and only if $\cD(I)$ is homogeneous with respect to a smaller group --and the latter never happens for exceptional varieties-- 
leads us to speculate that there exist no nestings for the Dynkin diagrams of exceptional type. 
Let us note that, for $\cD=\DG_2$ this conjecture follows from the fact that the existence of a section of the projection $\DG_2(1,2) \to \DG_2(i)$, which is a $\P^1$-bundle, would imply the decomposability of the associated rank two vector bundle (cf. Proposition \ref{prop:G2nest}). On the other hand, in our proof of Theorem \ref{thm:main} we find cohomological obstructions to the existence of nestings using explicit presentations of the cohomology rings of the varieties involved in terms of the Chern classes of the associated universal bundles; these presentations  are reasonably manageable and uniform in the classical cases.
We have not considered the exceptional cases of type $\DE_n$, $n=6,7,8$ and $\DF_4$ because the presentations of their cohomology rings are considerably more involved (see, for instance \cite{DZ}).

We finally remark that the projections $G/P\to G/P'$ are natural generalizations of the projectivizations of the universal bundles on Grassmannians (see Example \ref{ex:univbund}  and Remark \ref{rem:univbund}). Then the existence of nestings for $P \subset P' \subset G$ can be thought of as a reducibility condition for those bundles (Corollary \ref{cor:main}).  As a consequence, we get necessary and sufficient conditions for the existence of subbundles of the universal quotient bundle $\cQ$ on rational homogeneous spaces of classical type and Picard number one (Corollary \ref{cor:subbundles}). This result can be viewed as a generalization of \cite[Corollary 1.6]{DCR}.

\medskip

\noindent{\bf Description of the contents of the paper.} Section \ref{sec:prelim} contains background material on rational homogeneous varieties and bundles, and a discussion of the applications of Theorem \ref{thm:main} to reducibility of universal flag bundles and existence of subbundles of universal vector bundles.
In the last part of the section we present some technical lemmas regarding Chern classes of nef vector bundles, that we will use in the proof of Theorem \ref{thm:main}. 

The definition of nesting, together with some examples and properties, are presented in Section \ref{sec:nesting}; in particular, for each choice of $\cD$, $I$ and $J$, we define a scheme $\Nest(\cD,I,J)$ parametrizing nestings of the corresponding type. The fact that the group $G$ adjoint to $\cD$ acts on it allows us to extend the concept of nesting to rational homogeneous bundles over algebraic varieties. We will use this more general notion in the case in which the base variety is $\P^1$ (see Propositions \ref{prop:red2} and  \ref{prop:red3}).
In Section \ref{sec:special} we describe completely $\Nest(\DA_{2n-1},1,2n-1)$, $\Nest(\DB_3,1,3)$, and $\Nest(\DD_n,n-1,n)$, and show that they are quasiprojective homogeneous varieties. The proof of Theorem \ref{thm:main} is contained in Sections \ref{sec:reduc} and  \ref{sec:main}. First of all we reduce the problem to the case in which $I$ and $J$ consist of precisely one node and $I$ is extremal (Section \ref{sec:reduc}); then we complete the proof in the last section by studying the missing cases through cohomological computations.

\section{Notation and preliminaries}\label{sec:prelim}

\subsection{Rational homogeneous varieties}\label{ssec:prelim1}

Throughout the paper we will work over the field of complex numbers. We will recall here some basic notions on algebraic groups and their Lie algebras (see \cite{Hum1,Hum2}), and  rational homogeneous varieties (see \cite[Section~23.3]{FH}, where they are simply called ``homogeneous spaces'', and \cite{OttNotes}), and introduce the notation that we will use further on when dealing with them. 

Consider $G$ a semisimple algebraic group, and fix a  Borel subgroup $B\subset G$ and a maximal torus  $H \subset B$, i.e., a Cartan subgroup of $G$. The lattice of  characters  (respectively co-characters) of $H$ will be denoted by $\Mo(H):=\Hom(H,\C^*)$ (respectively $\Na(H):=\Hom(\C^*,H)$), and the Weyl group of $G$, defined as the quotient ${\rm N}(H)/H$ of the normalizer ${\rm N}(H)$ of $H$ in $G$, will be denoted by $W$. 

The choice of $H$ and $B$ defines a root system $\Phi\subset\Mo(H)$ and a base of positive simple roots $\Delta=\{\alpha_i|\,\, 1 \leq i \leq n\}$; the integer $\rk(G):=n$ is called the rank   of the group $G$. The induced action of $W$ on $\Mo(H)$ stabilizes $\Phi$, and for every element $\alpha\in\Phi$ there exists an element $s_\alpha\in W$ (called reflection with respect to $\alpha$) satisfying $s_\alpha^2=1$, and $s_\alpha(\alpha)=-\alpha$; moreover $W$ is generated by the elements $s_i:=s_{\alpha_i}$, $i=1,\dots,n$.  The Lie algebra of $G$, denoted by $\fg$, is completely determined by the Dynkin diagram $\cD$ of $G$; whenever $\fg$ is simple, the set of nodes  of $\cD$, that we will denote by $D$, will be numbered as in \cite[p.~58]{Hum3} (see Figure \ref{eq:dynkin} below). If $\cD$ is of type $\DA_n$, $\DB_n$, $\DC_n$ or $\DD_n$, $\fg$ and $G$ will be called of {\it classical type}. The other connected diagrams are called {\it exceptional}.

\def\lun{1.8}

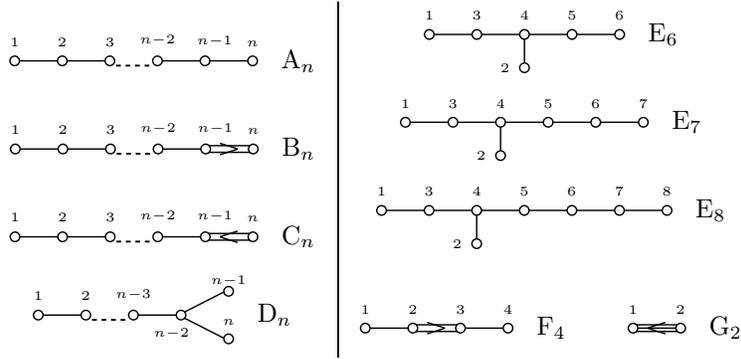
\begin{figure}[h!]
\begin{equation*}
\begin{array}{c|c}
\ifx\du\undefined
  \newlength{\du}
\fi
\setlength{\du}{\lun\unitlength}
\begin{tikzpicture}
\pgftransformxscale{1.000000}
\pgftransformyscale{1.000000}

\definecolor{dialinecolor}{rgb}{0.000000, 0.000000, 0.000000} 
\pgfsetstrokecolor{dialinecolor}
\definecolor{dialinecolor}{rgb}{0.000000, 0.000000, 0.000000} 
\pgfsetfillcolor{dialinecolor}


\pgfsetlinewidth{0.300000\du}
\pgfsetdash{}{0pt}
\pgfsetdash{}{0pt}

\pgfpathellipse{\pgfpoint{-6\du}{0\du}}{\pgfpoint{1\du}{0\du}}{\pgfpoint{0\du}{1\du}}
\pgfusepath{stroke}
\node at (-6\du,0\du){};

\pgfpathellipse{\pgfpoint{4\du}{0\du}}{\pgfpoint{1\du}{0\du}}{\pgfpoint{0\du}{1\du}}
\pgfusepath{stroke}
\node at (4\du,0\du){};

\pgfpathellipse{\pgfpoint{14\du}{0\du}}{\pgfpoint{1\du}{0\du}}{\pgfpoint{0\du}{1\du}}
\pgfusepath{stroke}
\node at (14\du,0\du){};

\pgfpathellipse{\pgfpoint{24\du}{0\du}}{\pgfpoint{1\du}{0\du}}{\pgfpoint{0\du}{1\du}}
\pgfusepath{stroke}
\node at (24\du,0\du){};

\pgfpathellipse{\pgfpoint{34\du}{0\du}}{\pgfpoint{1\du}{0\du}}{\pgfpoint{0\du}{1\du}}
\pgfusepath{stroke}
\node at (34\du,0\du){};

\pgfpathellipse{\pgfpoint{44\du}{0\du}}{\pgfpoint{1\du}{0\du}}{\pgfpoint{0\du}{1\du}}
\pgfusepath{stroke}
\node at (44\du,0\du){};

\pgfsetlinewidth{0.300000\du}
\pgfsetdash{}{0pt}
\pgfsetdash{}{0pt}
\pgfsetbuttcap

{\draw (-5\du,0\du)--(3\du,0\du);}
{\draw (5\du,0\du)--(13\du,0\du);}
{\draw (25\du,0\du)--(33\du,0\du);}
{\draw (35\du,0\du)--(43\du,0\du);}



\pgfsetlinewidth{0.400000\du}
\pgfsetdash{{1.000000\du}{1.000000\du}}{0\du}
\pgfsetdash{{1.000000\du}{1.000000\du}}{0\du}
\pgfsetbuttcap
{\draw (15.3\du,-1\du)--(23\du,-1\du);}

\node[anchor=west] at (48\du,0\du){${\rm A}_n$};

\node[anchor=south] at (-6\du,1.1\du){$\scriptscriptstyle 1$};

\node[anchor=south] at (4\du,1.1\du){$\scriptscriptstyle 2$};

\node[anchor=south] at (14\du,1.1\du){$\scriptscriptstyle 3$};

\node[anchor=south] at (24\du,1.1\du){$\scriptscriptstyle n-2$};

\node[anchor=south] at (36\du,1\du){$\scriptscriptstyle n-1$};

\node[anchor=south] at (44\du,1.1\du){$\scriptscriptstyle n$};

\end{tikzpicture} &\ifx\du\undefined
  \newlength{\du}
\fi
\setlength{\du}{\lun\unitlength}
\begin{tikzpicture}
\pgftransformxscale{1.000000}
\pgftransformyscale{1.000000}

\definecolor{dialinecolor}{rgb}{0.000000, 0.000000, 0.000000} 
\pgfsetstrokecolor{dialinecolor}
\definecolor{dialinecolor}{rgb}{0.000000, 0.000000, 0.000000} 
\pgfsetfillcolor{dialinecolor}


\pgfsetlinewidth{0.300000\du}
\pgfsetdash{}{0pt}
\pgfsetdash{}{0pt}

\pgfpathellipse{\pgfpoint{-6\du}{0\du}}{\pgfpoint{1\du}{0\du}}{\pgfpoint{0\du}{1\du}}
\pgfusepath{stroke}
\node at (-6\du,0\du){};

\pgfpathellipse{\pgfpoint{14\du}{-7\du}}{\pgfpoint{1\du}{0\du}}{\pgfpoint{0\du}{1\du}}
\pgfusepath{stroke}
\node at (14\du,-7\du){};

\pgfpathellipse{\pgfpoint{4\du}{0\du}}{\pgfpoint{1\du}{0\du}}{\pgfpoint{0\du}{1\du}}
\pgfusepath{stroke}
\node at (4\du,0\du){};

\pgfpathellipse{\pgfpoint{14\du}{0\du}}{\pgfpoint{1\du}{0\du}}{\pgfpoint{0\du}{1\du}}
\pgfusepath{stroke}
\node at (14\du,0\du){};

\pgfpathellipse{\pgfpoint{24\du}{0\du}}{\pgfpoint{1\du}{0\du}}{\pgfpoint{0\du}{1\du}}
\pgfusepath{stroke}
\node at (24\du,0\du){};

\pgfpathellipse{\pgfpoint{34\du}{0\du}}{\pgfpoint{1\du}{0\du}}{\pgfpoint{0\du}{1\du}}
\pgfusepath{stroke}
\node at (34\du,0\du){};

\pgfsetlinewidth{0.300000\du}
\pgfsetdash{}{0pt}
\pgfsetdash{}{0pt}
\pgfsetbuttcap

{\draw (-5\du,0\du)--(3\du,0\du);}
{\draw (5\du,0\du)--(13\du,0\du);}
{\draw (15\du,0\du)--(23\du,0\du);}
{\draw (25\du,0\du)--(33\du,0\du);}
{\draw (14\du,-1\du)--(14\du,-6\du);}

\node[anchor=west] at (38\du,0\du){${\rm E}_6$};

\node[anchor=south] at (-6\du,1.1\du){$\scriptscriptstyle 1$};

\node[anchor=south] at (4\du,1.1\du){$\scriptscriptstyle 3$};

\node[anchor=south] at (14\du,1.1\du){$\scriptscriptstyle 4$};

\node[anchor=south] at (24\du,1.1\du){$\scriptscriptstyle 5$};

\node[anchor=south] at (34\du,1.1\du){$\scriptscriptstyle 6$};

\node[anchor=east] at (12.9\du,-7\du){$\scriptscriptstyle 2$};

\end{tikzpicture} \\
\ifx\du\undefined
  \newlength{\du}
\fi
\setlength{\du}{\lun\unitlength}
\begin{tikzpicture}
\pgftransformxscale{1.000000}
\pgftransformyscale{1.000000}

\definecolor{dialinecolor}{rgb}{0.000000, 0.000000, 0.000000} 
\pgfsetstrokecolor{dialinecolor}
\definecolor{dialinecolor}{rgb}{0.000000, 0.000000, 0.000000} 
\pgfsetfillcolor{dialinecolor}


\pgfsetlinewidth{0.300000\du}
\pgfsetdash{}{0pt}
\pgfsetdash{}{0pt}

\pgfpathellipse{\pgfpoint{-6\du}{0\du}}{\pgfpoint{1\du}{0\du}}{\pgfpoint{0\du}{1\du}}
\pgfusepath{stroke}
\node at (-6\du,0\du){};

\pgfpathellipse{\pgfpoint{4\du}{0\du}}{\pgfpoint{1\du}{0\du}}{\pgfpoint{0\du}{1\du}}
\pgfusepath{stroke}
\node at (4\du,0\du){};

\pgfpathellipse{\pgfpoint{14\du}{0\du}}{\pgfpoint{1\du}{0\du}}{\pgfpoint{0\du}{1\du}}
\pgfusepath{stroke}
\node at (14\du,0\du){};

\pgfpathellipse{\pgfpoint{24\du}{0\du}}{\pgfpoint{1\du}{0\du}}{\pgfpoint{0\du}{1\du}}
\pgfusepath{stroke}
\node at (24\du,0\du){};

\pgfpathellipse{\pgfpoint{34\du}{0\du}}{\pgfpoint{1\du}{0\du}}{\pgfpoint{0\du}{1\du}}
\pgfusepath{stroke}
\node at (34\du,0\du){};

\pgfpathellipse{\pgfpoint{44\du}{0\du}}{\pgfpoint{1\du}{0\du}}{\pgfpoint{0\du}{1\du}}
\pgfusepath{stroke}
\node at (44\du,0\du){};

\pgfsetlinewidth{0.300000\du}
\pgfsetdash{}{0pt}
\pgfsetdash{}{0pt}
\pgfsetbuttcap

{\draw (-5\du,0\du)--(3\du,0\du);}
{\draw (5\du,0\du)--(13\du,0\du);}
{\draw (25\du,0\du)--(33\du,0\du);}
{\draw (34.65\du,0.7\du)--(43.35\du,0.7\du);}
{\draw (34.65\du,-0.7\du)--(43.35\du,-0.7\du);}


{\pgfsetcornersarced{\pgfpoint{0.300000\du}{0.300000\du}}\definecolor{dialinecolor}{rgb}{0.000000, 0.000000, 0.000000}
\pgfsetstrokecolor{dialinecolor}
\draw (37\du,-1.2\du)--(40.8\du,0\du)--(37\du,1.2\du);}

\pgfsetlinewidth{0.400000\du}
\pgfsetdash{{1.000000\du}{1.000000\du}}{0\du}
\pgfsetdash{{1.000000\du}{1.000000\du}}{0\du}
\pgfsetbuttcap
{\draw (15.3\du,-1\du)--(23\du,-1\du);}

\node[anchor=west] at (48\du,0\du){${\rm B}_n$};

\node[anchor=south] at (-6\du,1.1\du){$\scriptscriptstyle 1$};

\node[anchor=south] at (4\du,1.1\du){$\scriptscriptstyle 2$};

\node[anchor=south] at (14\du,1.1\du){$\scriptscriptstyle 3$};

\node[anchor=south] at (24\du,1.1\du){$\scriptscriptstyle n-2$};

\node[anchor=south] at (36\du,1.1\du){$\scriptscriptstyle n-1$};

\node[anchor=south] at (44\du,1.1\du){$\scriptscriptstyle n$};

\end{tikzpicture} &\ifx\du\undefined
  \newlength{\du}
\fi
\setlength{\du}{\lun\unitlength}
\begin{tikzpicture}
\pgftransformxscale{1.000000}
\pgftransformyscale{1.000000}

\definecolor{dialinecolor}{rgb}{0.000000, 0.000000, 0.000000} 
\pgfsetstrokecolor{dialinecolor}
\definecolor{dialinecolor}{rgb}{0.000000, 0.000000, 0.000000} 
\pgfsetfillcolor{dialinecolor}


\pgfsetlinewidth{0.300000\du}
\pgfsetdash{}{0pt}
\pgfsetdash{}{0pt}

\pgfpathellipse{\pgfpoint{-6\du}{0\du}}{\pgfpoint{1\du}{0\du}}{\pgfpoint{0\du}{1\du}}
\pgfusepath{stroke}
\node at (-6\du,0\du){};

\pgfpathellipse{\pgfpoint{14\du}{-7\du}}{\pgfpoint{1\du}{0\du}}{\pgfpoint{0\du}{1\du}}
\pgfusepath{stroke}
\node at (14\du,-7\du){};

\pgfpathellipse{\pgfpoint{4\du}{0\du}}{\pgfpoint{1\du}{0\du}}{\pgfpoint{0\du}{1\du}}
\pgfusepath{stroke}
\node at (4\du,0\du){};

\pgfpathellipse{\pgfpoint{14\du}{0\du}}{\pgfpoint{1\du}{0\du}}{\pgfpoint{0\du}{1\du}}
\pgfusepath{stroke}
\node at (14\du,0\du){};

\pgfpathellipse{\pgfpoint{24\du}{0\du}}{\pgfpoint{1\du}{0\du}}{\pgfpoint{0\du}{1\du}}
\pgfusepath{stroke}
\node at (24\du,0\du){};

\pgfpathellipse{\pgfpoint{34\du}{0\du}}{\pgfpoint{1\du}{0\du}}{\pgfpoint{0\du}{1\du}}
\pgfusepath{stroke}
\node at (34\du,0\du){};

\pgfpathellipse{\pgfpoint{44\du}{0\du}}{\pgfpoint{1\du}{0\du}}{\pgfpoint{0\du}{1\du}}
\pgfusepath{stroke}
\node at (44\du,0\du){};

\pgfsetlinewidth{0.300000\du}
\pgfsetdash{}{0pt}
\pgfsetdash{}{0pt}
\pgfsetbuttcap

{\draw (-5\du,0\du)--(3\du,0\du);}
{\draw (5\du,0\du)--(13\du,0\du);}
{\draw (15\du,0\du)--(23\du,0\du);}
{\draw (25\du,0\du)--(33\du,0\du);}
{\draw (35\du,0\du)--(43\du,0\du);}
{\draw (14\du,-1\du)--(14\du,-6\du);}

\node[anchor=west] at (48\du,0\du){${\rm E}_7$};

\node[anchor=south] at (-6\du,1.1\du){$\scriptscriptstyle 1$};

\node[anchor=south] at (4\du,1.1\du){$\scriptscriptstyle 3$};

\node[anchor=south] at (14\du,1.1\du){$\scriptscriptstyle 4$};

\node[anchor=south] at (24\du,1.1\du){$\scriptscriptstyle 5$};

\node[anchor=south] at (34\du,1.1\du){$\scriptscriptstyle 6$};

\node[anchor=south] at (44\du,1.1\du){$\scriptscriptstyle 7$};

\node[anchor=east] at (12.9\du,-7\du){$\scriptscriptstyle 2$};

\end{tikzpicture} \\
\ifx\du\undefined
  \newlength{\du}
\fi
\setlength{\du}{\lun\unitlength}
\begin{tikzpicture}
\pgftransformxscale{1.000000}
\pgftransformyscale{1.000000}

\definecolor{dialinecolor}{rgb}{0.000000, 0.000000, 0.000000} 
\pgfsetstrokecolor{dialinecolor}
\definecolor{dialinecolor}{rgb}{0.000000, 0.000000, 0.000000} 
\pgfsetfillcolor{dialinecolor}


\pgfsetlinewidth{0.300000\du}
\pgfsetdash{}{0pt}
\pgfsetdash{}{0pt}

\pgfpathellipse{\pgfpoint{-6\du}{0\du}}{\pgfpoint{1\du}{0\du}}{\pgfpoint{0\du}{1\du}}
\pgfusepath{stroke}
\node at (-6\du,0\du){};

\pgfpathellipse{\pgfpoint{4\du}{0\du}}{\pgfpoint{1\du}{0\du}}{\pgfpoint{0\du}{1\du}}
\pgfusepath{stroke}
\node at (4\du,0\du){};

\pgfpathellipse{\pgfpoint{14\du}{0\du}}{\pgfpoint{1\du}{0\du}}{\pgfpoint{0\du}{1\du}}
\pgfusepath{stroke}
\node at (14\du,0\du){};

\pgfpathellipse{\pgfpoint{24\du}{0\du}}{\pgfpoint{1\du}{0\du}}{\pgfpoint{0\du}{1\du}}
\pgfusepath{stroke}
\node at (24\du,0\du){};

\pgfpathellipse{\pgfpoint{34\du}{0\du}}{\pgfpoint{1\du}{0\du}}{\pgfpoint{0\du}{1\du}}
\pgfusepath{stroke}
\node at (34\du,0\du){};

\pgfpathellipse{\pgfpoint{44\du}{0\du}}{\pgfpoint{1\du}{0\du}}{\pgfpoint{0\du}{1\du}}
\pgfusepath{stroke}
\node at (44\du,0\du){};

\pgfsetlinewidth{0.300000\du}
\pgfsetdash{}{0pt}
\pgfsetdash{}{0pt}
\pgfsetbuttcap

{\draw (-5\du,0\du)--(3\du,0\du);}
{\draw (5\du,0\du)--(13\du,0\du);}
{\draw (25\du,0\du)--(33\du,0\du);}
{\draw (34.65\du,0.7\du)--(43.35\du,0.7\du);}
{\draw (34.65\du,-0.7\du)--(43.35\du,-0.7\du);}


{\pgfsetcornersarced{\pgfpoint{0.300000\du}{0.300000\du}}\definecolor{dialinecolor}{rgb}{0.000000, 0.000000, 0.000000}
\pgfsetstrokecolor{dialinecolor}
\draw (40.8\du,-1.2\du)--(37\du,0\du)--(40.8\du,1.2\du);}

\pgfsetlinewidth{0.400000\du}
\pgfsetdash{{1.000000\du}{1.000000\du}}{0\du}
\pgfsetdash{{1.000000\du}{1.00000\du}}{0\du}
\pgfsetbuttcap
{\draw (15.3\du,-1\du)--(23\du,-1\du);}

\node[anchor=west] at (48\du,0\du){${\rm C}_n$};

\node[anchor=south] at (-6\du,1.1\du){$\scriptscriptstyle 1$};

\node[anchor=south] at (4\du,1.1\du){$\scriptscriptstyle 2$};

\node[anchor=south] at (14\du,1.1\du){$\scriptscriptstyle 3$};

\node[anchor=south] at (24\du,1.1\du){$\scriptscriptstyle n-2$};

\node[anchor=south] at (36\du,1.1\du){$\scriptscriptstyle n-1$};

\node[anchor=south] at (44\du,1.1\du){$\scriptscriptstyle n$};

\end{tikzpicture} &\ifx\du\undefined
  \newlength{\du}
\fi
\setlength{\du}{\lun\unitlength}
\begin{tikzpicture}
\pgftransformxscale{1.000000}
\pgftransformyscale{1.000000}

\definecolor{dialinecolor}{rgb}{0.000000, 0.000000, 0.000000} 
\pgfsetstrokecolor{dialinecolor}
\definecolor{dialinecolor}{rgb}{0.000000, 0.000000, 0.000000} 
\pgfsetfillcolor{dialinecolor}


\pgfsetlinewidth{0.300000\du}
\pgfsetdash{}{0pt}
\pgfsetdash{}{0pt}

\pgfpathellipse{\pgfpoint{-6\du}{0\du}}{\pgfpoint{1\du}{0\du}}{\pgfpoint{0\du}{1\du}}
\pgfusepath{stroke}
\node at (-6\du,0\du){};

\pgfpathellipse{\pgfpoint{14\du}{-7\du}}{\pgfpoint{1\du}{0\du}}{\pgfpoint{0\du}{1\du}}
\pgfusepath{stroke}
\node at (14\du,-7\du){};

\pgfpathellipse{\pgfpoint{4\du}{0\du}}{\pgfpoint{1\du}{0\du}}{\pgfpoint{0\du}{1\du}}
\pgfusepath{stroke}
\node at (4\du,0\du){};

\pgfpathellipse{\pgfpoint{14\du}{0\du}}{\pgfpoint{1\du}{0\du}}{\pgfpoint{0\du}{1\du}}
\pgfusepath{stroke}
\node at (14\du,0\du){};

\pgfpathellipse{\pgfpoint{24\du}{0\du}}{\pgfpoint{1\du}{0\du}}{\pgfpoint{0\du}{1\du}}
\pgfusepath{stroke}
\node at (24\du,0\du){};

\pgfpathellipse{\pgfpoint{34\du}{0\du}}{\pgfpoint{1\du}{0\du}}{\pgfpoint{0\du}{1\du}}
\pgfusepath{stroke}
\node at (34\du,0\du){};

\pgfpathellipse{\pgfpoint{44\du}{0\du}}{\pgfpoint{1\du}{0\du}}{\pgfpoint{0\du}{1\du}}
\pgfusepath{stroke}
\node at (44\du,0\du){};

\pgfpathellipse{\pgfpoint{54\du}{0\du}}{\pgfpoint{1\du}{0\du}}{\pgfpoint{0\du}{1\du}}
\pgfusepath{stroke}
\node at (54\du,0\du){};

\pgfsetlinewidth{0.300000\du}
\pgfsetdash{}{0pt}
\pgfsetdash{}{0pt}
\pgfsetbuttcap

{\draw (-5\du,0\du)--(3\du,0\du);}
{\draw (5\du,0\du)--(13\du,0\du);}
{\draw (15\du,0\du)--(23\du,0\du);}
{\draw (25\du,0\du)--(33\du,0\du);}
{\draw (35\du,0\du)--(43\du,0\du);}
{\draw (45\du,0\du)--(53\du,0\du);}
{\draw (14\du,-1\du)--(14\du,-6\du);}

\node[anchor=west] at (58\du,0\du){${\rm E}_8$};

\node[anchor=south] at (-6\du,1.1\du){$\scriptscriptstyle 1$};

\node[anchor=south] at (4\du,1.1\du){$\scriptscriptstyle 3$};

\node[anchor=south] at (14\du,1.1\du){$\scriptscriptstyle 4$};

\node[anchor=south] at (24\du,1.1\du){$\scriptscriptstyle 5$};

\node[anchor=south] at (34\du,1.1\du){$\scriptscriptstyle 6$};

\node[anchor=south] at (44\du,1.1\du){$\scriptscriptstyle 7$};

\node[anchor=south] at (54\du,1.1\du){$\scriptscriptstyle 8$};

\node[anchor=east] at (12.9\du,-7\du){$\scriptscriptstyle 2$};

\end{tikzpicture} \\\ifx\du\undefined
  \newlength{\du}
\fi
\setlength{\du}{\lun\unitlength}
\begin{tikzpicture}
\pgftransformxscale{1.000000}
\pgftransformyscale{1.000000}

\definecolor{dialinecolor}{rgb}{0.000000, 0.000000, 0.000000} 
\pgfsetstrokecolor{dialinecolor}
\definecolor{dialinecolor}{rgb}{0.000000, 0.000000, 0.000000} 
\pgfsetfillcolor{dialinecolor}


\pgfsetlinewidth{0.300000\du}
\pgfsetdash{}{0pt}
\pgfsetdash{}{0pt}

\pgfpathellipse{\pgfpoint{-6\du}{0\du}}{\pgfpoint{1\du}{0\du}}{\pgfpoint{0\du}{1\du}}
\pgfusepath{stroke}
\node at (-6\du,0\du){};

\pgfpathellipse{\pgfpoint{4\du}{0\du}}{\pgfpoint{1\du}{0\du}}{\pgfpoint{0\du}{1\du}}
\pgfusepath{stroke}
\node at (4\du,0\du){};

\pgfpathellipse{\pgfpoint{14\du}{0\du}}{\pgfpoint{1\du}{0\du}}{\pgfpoint{0\du}{1\du}}
\pgfusepath{stroke}
\node at (14\du,0\du){};

\pgfpathellipse{\pgfpoint{24\du}{0\du}}{\pgfpoint{1\du}{0\du}}{\pgfpoint{0\du}{1\du}}
\pgfusepath{stroke}
\node at (24\du,0\du){};

\pgfpathellipse{\pgfpoint{34\du}{5\du}}{\pgfpoint{1\du}{0\du}}{\pgfpoint{0\du}{1\du}}
\pgfusepath{stroke}
\node at (34\du,5\du){};

\pgfpathellipse{\pgfpoint{34\du}{-5\du}}{\pgfpoint{1\du}{0\du}}{\pgfpoint{0\du}{1\du}}
\pgfusepath{stroke}
\node at (34\du,-5\du){};

\pgfsetlinewidth{0.300000\du}
\pgfsetdash{}{0pt}
\pgfsetdash{}{0pt}
\pgfsetbuttcap

{\draw (-5\du,0\du)--(3\du,0\du);}
{\draw (15\du,0\du)--(23\du,0\du);}
{\draw (24.65\du,0.6\du)--(33.1\du,4.9\du);}
{\draw (24.65\du,-0.6\du)--(33.1\du,-4.9\du);}



\pgfsetlinewidth{0.400000\du}
\pgfsetdash{{1.000000\du}{1.000000\du}}{0\du}
\pgfsetdash{{1.000000\du}{1.000000\du}}{0\du}
\pgfsetbuttcap
{\draw (5.3\du,-1\du)--(13\du,-1\du);}

\node[anchor=west] at (38\du,0\du){${\rm D}_n$};

\node[anchor=south] at (-6\du,1.1\du){$\scriptscriptstyle 1$};

\node[anchor=south] at (4\du,1.1\du){$\scriptscriptstyle 2$};

\node[anchor=south] at (14\du,1.1\du){$\scriptscriptstyle n-3$};

\node[anchor=south] at (22\du,-7.1\du){$\scriptscriptstyle n-2$};

\node[anchor=south] at (34\du,4.1\du){$\scriptscriptstyle n-1$};

\node[anchor=south] at (34\du,-4.1\du){$\scriptscriptstyle n$};

\end{tikzpicture} &\ifx\du\undefined
  \newlength{\du}
\fi
\setlength{\du}{\lun\unitlength}
\begin{tikzpicture}
\pgftransformxscale{1.000000}
\pgftransformyscale{1.000000}

\definecolor{dialinecolor}{rgb}{0.000000, 0.000000, 0.000000} 
\pgfsetstrokecolor{dialinecolor}
\definecolor{dialinecolor}{rgb}{0.000000, 0.000000, 0.000000} 
\pgfsetfillcolor{dialinecolor}


\pgfsetlinewidth{0.300000\du}
\pgfsetdash{}{0pt}
\pgfsetdash{}{0pt}

\pgfpathellipse{\pgfpoint{-6\du}{0\du}}{\pgfpoint{1\du}{0\du}}{\pgfpoint{0\du}{1\du}}
\pgfusepath{stroke}
\node at (-6\du,0\du){};

\pgfpathellipse{\pgfpoint{4\du}{0\du}}{\pgfpoint{1\du}{0\du}}{\pgfpoint{0\du}{1\du}}
\pgfusepath{stroke}
\node at (4\du,0\du){};

\pgfpathellipse{\pgfpoint{14\du}{0\du}}{\pgfpoint{1\du}{0\du}}{\pgfpoint{0\du}{1\du}}
\pgfusepath{stroke}
\node at (14\du,0\du){};

\pgfpathellipse{\pgfpoint{24\du}{0\du}}{\pgfpoint{1\du}{0\du}}{\pgfpoint{0\du}{1\du}}
\pgfusepath{stroke}
\node at (24\du,0\du){};

\pgfsetlinewidth{0.300000\du}
\pgfsetdash{}{0pt}
\pgfsetdash{}{0pt}
\pgfsetbuttcap

{\draw (-5\du,0\du)--(3\du,0\du);}
{\draw (15\du,0\du)--(23\du,0\du);}
{\draw (4.65\du,0.7\du)--(13.35\du,0.7\du);}
{\draw (4.65\du,-0.7\du)--(13.35\du,-0.7\du);}


{\pgfsetcornersarced{\pgfpoint{0.300000\du}{0.300000\du}}\definecolor{dialinecolor}{rgb}{0.000000, 0.000000, 0.000000}
\pgfsetstrokecolor{dialinecolor}
\draw (7\du,-1.2\du)--(10.8\du,0\du)--(7\du,1.2\du);}

\node[anchor=west] at (28\du,0\du){${\rm F}_4$};

\node[anchor=south] at (-6\du,1.1\du){$\scriptscriptstyle 1$};

\node[anchor=south] at (4\du,1.1\du){$\scriptscriptstyle 2$};

\node[anchor=south] at (14\du,1.1\du){$\scriptscriptstyle 3$};

\node[anchor=south] at (24\du,1.1\du){$\scriptscriptstyle 4$};

\end{tikzpicture} \hspace{0.6cm}
\ifx\du\undefined
  \newlength{\du}
\fi
\setlength{\du}{\lun\unitlength}
\begin{tikzpicture}
\pgftransformxscale{1.000000}
\pgftransformyscale{1.000000}

\definecolor{dialinecolor}{rgb}{0.000000, 0.000000, 0.000000} 
\pgfsetstrokecolor{dialinecolor}
\definecolor{dialinecolor}{rgb}{0.000000, 0.000000, 0.000000} 
\pgfsetfillcolor{dialinecolor}


\pgfsetlinewidth{0.300000\du}
\pgfsetdash{}{0pt}
\pgfsetdash{}{0pt}

\pgfpathellipse{\pgfpoint{-6\du}{0\du}}{\pgfpoint{1\du}{0\du}}{\pgfpoint{0\du}{1\du}}
\pgfusepath{stroke}
\node at (-6\du,0\du){};

\pgfpathellipse{\pgfpoint{4\du}{0\du}}{\pgfpoint{1\du}{0\du}}{\pgfpoint{0\du}{1\du}}
\pgfusepath{stroke}
\node at (4\du,0\du){};

\pgfsetlinewidth{0.300000\du}
\pgfsetdash{}{0pt}
\pgfsetdash{}{0pt}
\pgfsetbuttcap

{\draw (-5\du,0\du)--(3\du,0\du);}
{\draw (-5.35\du,0.7\du)--(3.35\du,0.7\du);}
{\draw (-5.35\du,-0.7\du)--(3.35\du,-0.7\du);}


{\pgfsetcornersarced{\pgfpoint{0.300000\du}{0.300000\du}}\definecolor{dialinecolor}{rgb}{0.000000, 0.000000, 0.000000}
\pgfsetstrokecolor{dialinecolor}
\draw (0.8\du,-1.2\du)--(-3\du,0\du)--(0.8\du,1.2\du);}

\node[anchor=west] at (8\du,0\du){${\rm G}_2$};

\node[anchor=south] at (-6\du,1.1\du){$\scriptscriptstyle 1$};

\node[anchor=south] at (4\du,1.1\du){$\scriptscriptstyle 2$};

\end{tikzpicture} 
\end{array}
\end{equation*}
\caption{Numbering of connected Dynkin diagrams. }\label{eq:dynkin}
\end{figure}

\medskip

Given $\cD$, there exists precisely one semisimple group $G$, called the adjoint group of $\cD$,  whose lattice of characters coincides with the one generated by $\Phi$. This group is the image of any other semisimple group with Lie algebra $\fg$ via its adjoint representation. The rest of the semisimple groups with Lie algebra $\fg$ are isogenous to the adjoint group. 
For instance, the adjoint group of type $\DA_n$ is the projective linear group $\PGL(n+1)$, and the adjoint groups of types $\DB_n$, $\DC_n$ and $\DD_n$ are the images of the natural maps:
$$
\SO(2n+1)\to\PGL(2n+1),\quad \SP(2n)\to\PGL(2n),\quad\SO(2n)\to\PGL(2n),
$$
denoted, respectively,  by $\PSO(2n+1)$, $\PSP(2n)$, and $\PSO(2n)$. Note that, in the case $\DB_n$, the map $\SO(2n+1)\to\PSO(2n+1)$ is an isomorphism.   

For every subset of nodes $I \subset D$  one can construct a parabolic subgroup $P(D \setminus I) \subset G$ as $P(D\setminus I)=BW(D \setminus I)B$, where $W(D \setminus I)$ is the subgroup of $W$ generated by the reflections $s_i$ associated with the indices $i \notin I$. The quotient $G/P(D \setminus I)$, called a rational homogeneous variety, depends only on the Dynkin diagram $\cD$ (that is, on the Lie algebra $\fg$) and on the set of nodes $I\subset D$. Thus the variety $G/P(D \setminus I)$ is commonly represented by the Dynkin diagram $\cD$ marked in the nodes of $I$,  and we will write $$\cD(I):=G/P(D \setminus I).$$ Sometimes, given a set of nodes $I=\{i_1, \dots, i_s\}$, we will write $\cD(i_1, \dots, i_s):=\cD(I)$. Any projective quotient of $G$, or of any other semisimple  group $G'$ isogenous to $G$, is isomorphic to one of these varieties. For $I=D$ we get $\cD(D)=G/P(\emptyset)=G/B$, known as the complete flag variety associated with $G$. On the other hand $\cD(\emptyset)=G/G$ is a point. In the cases in which $\cD$ is disconnected, a rational homogeneous variety $\cD(I)$ is a product, whose factors correspond to the connected components $\cD_i$ of $\cD$, marked on the nodes of $I$ contained in $\cD_i$. 

\begin{remark}\label{rem:special}
Restricting ourselves to the case in which $\cD$ is connected,  a rational homogeneous variety determines uniquely a marked Dynkin diagram $\cD(I)$, up to isomorphism, with three exceptions (cf. \cite[Ch.~3, Theorem 2]{Akh}): projective spaces of dimension $2n-1$ ($n\geq 2$), the spinor varieties parametrizing linear subspaces of maximal dimension on a $(2n-2)$-dimensional quadric $(n\geq 4$), and the smooth $5$-dimensional quadric $Q^5$. In fact we may write:
$$\begin{array}{c}\vspace{0.2cm}
\P^{2n-1}=\DA_{2n-1}(1)=\DA_{2n-1}(2n-1)=\DC_{n}(1),\\\vspace{0.2cm} \DD_n(n-1)=\DD_n(n)=\DB_{n-1}(n-1),\qquad
Q^5=\DB_3(1)=\DG_2(1).
\end{array}
$$ 
\end{remark}

Given two nonempty disjoint subsets of nodes $I,J$, the inclusion $I \subset  I\cup J $ defines a morphism $\pi_{I\cup J,I}:\cD(I\cup J) \to\cD(I)$ (called unmarking of $J$), that can be shown to be a smooth fiber type contraction; the fibers of $\pi_{I\cup J,I}$ are rational homogeneous varieties of type $\ol{\cD}(J)$, where $\ol{\cD}$ denotes the Dynkin diagram obtained from $\cD$ by deleting the nodes of $I$. In the particular case of $J=\{j\}$, $I=D \setminus \{j\}$, one gets an elementary contraction $\rho_j:=\pi_{D,D\setminus\{j\}}:\cD(D) \to \cD(D \setminus \{j\})$, whose fibers are isomorphic to $\P^1$. Denoting by $\Gamma_j$ the numerical class of the fibers of $\rho_j$, and by $K_j$ the corresponding relative canonical bundle, for every $j\in D$, the matrix of intersections $(-K_i\cdot \Gamma_j)$ is equal to the Cartan matrix of $G$ (cf. \cite[Proposition 3]{MOSW}), which encodes the information of the adjacencies of $\cD$. The nodes $i$ and $j$ are joined by $(-K_j \cdot \Gamma_i)(-K_i \cdot \Gamma_ j)$ edges, and when two nodes $i$ and $j$ are joined by a double or triple edge, we add to it an arrow, pointing to $i$ if $-K_i\cdot \Gamma_j >-K_j\cdot \Gamma_i$.

The vector space $\NU(G/B)$ of real linear combinations of line bundles modulo numerical equivalence on the complete flag variety $G/B$ is isomorphic to the $\rk(G)$--dimensional vector space $\Mo(H)\otimes_\Z\R$, spanned by the characters. Following the Shephard--Todd--Chevalley Theorem (cf. \cite[Theorem~3.1]{Hi82}), within the symmetric algebra $S\NU(G/B)$, the subalgebra of $W$-invariant polynomials $S\NU(G/B)^W$ is a polynomial ring, generated by $\rk(G)$ algebraically independent homogeneous polynomials, whose degrees, called fundamental degrees of $G$ (see Table \ref{tab:fundeg}), depend only on $\cD$. The maximum of the fundamental degrees of $G$ is known as its Coxeter number; we will denote the Coxeter number of $\cD$ by $\Cox(\cD)$.
\renewcommand*{\arraystretch}{1.1}
\begin{table}[h!]
\centering
\begin{tabular}{r|l}
\hline
Group type&Fundamental degrees\\\hline
$\DA_n$&
$2,3,\dots,n,n+1$\\
$\DB_n$,$\DC_n$&
$2,4,\dots,2(n-1),2n$\\
$\DD_n$&
$2,4,\dots,2(n-1),n$\\
$\DE_6$&
$2,5,6,8,9,12$\\
$\DE_7$&
$2,6,8,10,12,14,18$\\
$\DE_8$&
$2,8,12,14,18,20,24,30$\\
$\DF_4$&
$2,6,8,12$\\
$\DG_2$&
$2,6$\\\hline
\end{tabular} 
\caption{Fundamental degrees of the semisimple algebraic groups with simple Lie algebra.}
\label{tab:fundeg}
\end{table}

Furthermore, the well known Bern\v{s}te\u{\i}n--Gel'fand--Gel'fand theorem (cf. \cite{BGG}) tells us that the real cohomology ring $$\HH^\bullet(G/B):=\HH^\bullet(G/B,\R)$$ can be written as a quotient of $S\NU(G/B)$ by the ideal generated by homogeneous $W$-invariant polynomials of positive degree, $S\NU(G/B)^W_+\subset S\NU(G/B)^W$, (see \cite[Section~3]{MOS6} and the reference therein). This leads also to a description of the cohomology rings $\HH^\bullet(\cD(I)):=\HH^\bullet(\cD(I),\R)$, which are the invariant subalgebras of $\HH^\bullet(G/B)$ by the action of the subgroups $W(D\setminus I)$. In Section \ref{sec:main} we will use explicit presentations (cf. Tables \ref{tab:gen} and \ref{tab:pres})  
of these cohomology rings, for certain choices of $\cD(I)$, to complete the proof of Theorem \ref{thm:main}. 

\subsection{Rational homogeneous bundles}

Let $X$ be a complex algebraic variety, that we will assume to be simply connected  (which is the case, for instance, whenever $X$ is rational homogeneous or, more generally, a smooth Fano variety), and $F$ be a rational homogeneous variety. Denoting by $G$ the identity component of the automorphism group of $F$, which is known to be semisimple (cf. \cite[Thm. 3.11]{Huck}), and by $\cD$ the Dynkin diagram of $G$, we may write, as in the previous section, $F=\cD(I)$, for a set of nodes $I$ of $\cD$. A smooth morphism $\pi:Y \to X$ such that all its fibers are isomorphic to $F$ is called an $F$-bundle or, in general, a rational homogeneous bundle over $X$. Following a theorem of Fischer and Grauert (see \cite[p. 29]{BPVV}), $\pi$ is locally trivial in the analytic topology, hence, the simple connectedness of $X$ implies that $\pi$ is determined by a $1$-cocycle $\theta \in \HH^1(X,G)$, where $\HH^1(X,G)$ denotes the {\v C}ech cohomology of the  sheafified group $G$ on the analytic space associated with $X$.  

Conversely, any cocycle $\theta\in \HH^1(X,G)$  defines a $G$-principal bundle $\pi_G:\cE\to X$; given any set of nodes $J$ of $\cD$, it leads to a $\cD(J)$-bundle on $X$, by setting $$Y_J:=\cE\times^G\cD(J)=(\cE\times \cD(J))/\sim,\qquad (e,gP(D\setminus J))\sim(eh,h^{-1}gP(D\setminus J)),\,\,\forall h\in G. $$ 
The projection $\pi_J:Y_J\to X$ is defined as the natural map sending the class of $(e,gP(D\setminus J))$ to $\pi_G(e)$. If $\theta$ is the cocycle defined by a $\cD(I)$-bundle $\pi:Y\to X$ as above, then the bundle $\pi_I:Y_I\to X$ coincides with it. 

Furthermore, this construction is compatible with the contractions of rational homogeneous varieties described in the previous sections: given two disjoint sets of nodes $I,J\subset D$, the inclusion $I\subset I\cup J$, together with the cocycle $\theta$ provides a contraction: 
$$
\pi_{I\cup J,I}:Y_{I\cup J}\to Y_I 
$$
satisfying $\pi_I\circ\pi_{I\cup J,I}=\pi_{I\cup J}$. Fiberwise over points of $X$, $\pi_{I\cup J,I}$ is equal to the contraction $\cD(I\cup J)\to \cD(I)$, hence $\pi_{I\cup J,I}$ is a $\ol{\cD}(J)$-bundle, where $\ol{\cD}$ denotes the Dynkin diagram obtained from $\cD$ by deleting the nodes of $I$. 

In particular, all the bundles constructed upon the cocycle $\theta$ can be obtained via contractions of the $\cD(D)$-bundle $\pi_D:Y_D\to X$, so that, denoting $\rho_{D\setminus I}:=\pi_{D,I}$, we have commutative diagrams:
$$
\xymatrix{Y\ar@/^1pc/[rr]^{\pi_D}\ar[r]_{\rho_{D\setminus I}}&Y_I\ar[r]_{\pi_I}&X}
$$
When we do not want to specify the particular diagram $\cD$ we are considering, we will usually refer to a $\cD(D)$-bundle as a flag bundle over $X$; for instance, the contractions $\rho_{D\setminus I}$ are flag bundles over $Y_I$.

Particularly interesting examples of rational homogeneous bundles are the so-called universal bundles, obtained by considering contractions of rational homogeneous varieties: 

\begin{example}\label{ex:univbund} 
Let $\cD$ be a Dynkin diagram, $Y:=\cD(D)$ be the corresponding complete flag variety, $I \subset D$ be a subset of nodes of $\cD$, and $X:=\cD(I)$ be the corresponding rational homogeneous variety. The contraction $\pi:=\rho_{D\setminus I}:Y \to X$, corresponding to unmarking the nodes in $D\setminus I$, is a flag bundle: its fibers are complete flag varieties with respect to the Dynkin diagram $\ol{\cD}$, obtained from $\cD$ by deleting the nodes of $I$. Let us write $D\setminus I=D_1\cup\dots\cup D_k$, where each $D_s$ denotes the set of nodes of a connected component $\ol{\cD}_s$ of $\ol{\cD}$, so that the contraction $\rho_{D\setminus I}$ equals the fiber product of the contractions $$\pi_{I\cup D_s,I}:\cD(I\cup D_s) \to X.$$ The contractions $\pi_{I\cup D_s,I}$,  $s=1,\dots,k$, whose fibers are complete flag varieties $\ol{\cD}_s(D_s)$, are called {\em universal flag bundles} over $X$. Moreover, for every subset $J_s\subset D_s$, the corresponding intermediate contraction $$\pi_{I\cup J_s,I}:\cD(I\cup J_s) \to X$$ will be called a {\em universal rational homogeneous bundle} associated with $X$. 
\end{example}

\begin{remark}\label{rem:univbund} 
In the case in which $\cD(I)=\DA_n(r)$, $r\in \{1,\dots,n\}$, that is in the case in which $\cD(I)$ is a Grassmannian of $\P^{r-1}$'s in $\P^n=\P(V)$, the universal flag bundles are the complete flag bundles associated to the  projectivizations of the corresponding universal quotient bundle $\cQ$ and universal subbundle $\cS$, fitting in the 
short exact sequence:
$$
\shse{\cS^\vee}{\cO_{\DA_n(r)}\otimes V}{\cQ}.
$$
If $\cD(I)=\DB_n(r),\DC_n(r),\DD_n(r)$, universal flag bundles can be constructed as follows.
Let $V$ be the natural representation of a Lie algebra of type $\cD = \DB_n$, $\DC_n$, $\DD_n$ (which has dimension $N=2n+1,2n$, and $2n$, respectively), and let $\P(V)$ be its 
projectivization. 
Every variety $\cD(r)$ can be embedded in a Grassmannian $\G(k,\P(V))$ of $k$-dimensional projective subspaces in $\P(V)$, where $k=r-1$ if $\cD(r) \not = \DD_n(n-1)$ and $k=n-1$ if $\cD(r) = \DD_n(n-1)$. 
We may then consider the restrictions to $\cD(r)$ of the corresponding universal bundles (abusing notation, we denote them by $\cQ$ and $\cS$).  In the cases of $\DB_n$, $\DD_n$ (respectively, $\DC_n$), $V$ supports a nondegenerate quadratic (respectively, skew-symmetric) form $\omega:V^\vee\to V$, with respect to which the vector subspaces parametrized by $\cD(r)$ are isotropic. In other words, the composition:
$$
\cQ^\vee\lra\cO_{\cD(r)}\otimes V^\vee\stackrel{\omega}{\lra} \cO_{\cD(r)}\otimes V\lra \cQ
$$
is zero, so we have an exact sequence 
\begin{equation}\label{eq:KSQ}
\shse{\cK}{\cS}{\cQ}
\end{equation}
and the form $\omega$ defines a symmetric (resp. skew-symmetric) isomorphism $\cK\simeq \cK^\vee$ (see \cite[Section 3.3]{MOS6}). One of the universal flag bundles on $\cD(r)$ is the complete flag bundle associated to $\P(\cQ)$; the other is, in each case, a flag bundle of type $\DB_{n-r}$, $\DC_{n-r}$ or $\DD_{n-r}$, constructed upon the $\P(\cK)$ and the nondegenerate symmetric or antisymmetric form defined above.
\end{remark}

We finish this section by discussing the consequences of Theorem \ref{thm:main} on the reducibility of universal flag bundles. Let us recall the following definition (see \cite[Section~3]{MOS5} for further details):

\begin{definition}\label{def:reducible}
A $\cD(D)$-bundle $\pi_D:Y_D\to X$ over an algebraic variety $X$ 
is called {\it reducible with respect to }$I\subset D$ if the corresponding $\cD(I)$-bundle $\pi_I:Y_I\to X$ admits a section $s_I:X\to Y_I$.
\end{definition}
\begin{remark}\label{rem:reducible} 
A $\cD(D)$-bundle as above is reducible if and only if its defining cocycle $\theta\in \HH^1(X,G)$ lies in the image of the natural map $\HH^1(X,P(D\setminus I))\to \HH^1(X,G)$. In the case in which $G=\SL(V)$, parabolic subgroups are subgroups of block-triangular matrices, hence saying that a $\cD(D)$-bundle is reducible with respect to some $I\neq\emptyset$ is equivalent to saying that the corresponding vector bundle $\cE\times^GV$ obtained upon the natural representation $V$ of $\SL(V)$ admits a nontrivial subbundle. 
\end{remark}

This observation already allows us to study the existence of nestings in the cases in which $\cD(I\cup J)\to \cD(I)$ is a $\P^1$-bundle. For instance, we may write the following statement: 

\begin{proposition}\label{prop:G2nest}
There exist no nestings of homogeneous $\DG_2$-varieties.
\end{proposition}

\begin{proof}
The three homogeneous $\DG_2$-varieties fit in the following diagram:
$$
\xymatrix{&\DG_2(1,2)\ar[rd]^{\rho_1}\ar[ld]_{\rho_2}&\\\DG_2(1)&&\DG_2(2)}
$$
where the two contractions are $\P^1$-bundles, so that we may write $\DG_2(1,2)\simeq\P(\cC)\simeq\P(\cQ)$, where $\cC$ is the Cayley bundle on $\DG_2(1)\simeq Q^5$, and $\cQ$ is the universal bundle on $\DG_2(2)$ (see \cite{O,MOS3}). The existence of a nesting of type $(\DG_2,1,2)$ 
would translate into the reducibility of $\rho_2$, 
that is into the existence of a line subbundle  $\cL\hookrightarrow\cC$. Since $\DG_2(1)$ is Fano of Picard number one, Kodaira vanishing theorem would tell us that $\cC$ would split as $\cL\oplus\cL'$, where $\cL'$ denotes the cokernel of the inclusion of $\cL$ into $\cC$; this contradicts the fact that  $\rho_1:\P(\cC)\to \DG_2(2)$ is a $\P^1$-bundle.
In fact, since  $\Pic(\DG_2(1))=\Z H$, for an ample line bundle $H$, we could write $\cC\simeq H^{\otimes a} \oplus H^{\otimes b}$, for $a\leq b\in \Z$, and $\DG_2(1,2)\simeq \P(\cO\oplus H^{\otimes(b-a)})$, which could only have a fiber type contraction different from $\rho_2$ in the case $b=a$, that is, if $\DG_2(1,2)\simeq\P^1\times \DG_2(1)$. 

A similar argument holds  for the nestings of type $(\DG_2,2,1)$. 
\end{proof}

In the language of rational homogeneous bundles, the main result of this paper, Theorem \ref{thm:main}, reads as follows:

\begin{corollary}\label{cor:main}
Let $\cD(I)$ be a rational homogeneous space of classical type, and $\pi:Y=\cD(I\cup D_s)\to \cD(I)$ be a universal flag bundle on it. Then $Y$ is reducible if and only if $(\cD,I)$ is isomorphic to  one of the following:
$$
(\DA_{2n-1},1)\,\,\,\, \mbox{$n\geq 2$},\quad (\DB_3,1),\quad (\DD_n,n) \,\,\,\,n\geq 4.
$$
\end{corollary}

In particular, we may apply this result to the universal rational homogeneous bundles on varieties of type $\cD(r)$ with $\cD$ of classical type (see Remark \ref{rem:univbund}), generalizing  \cite[Corollary 1.6]{DCR}.

\begin{corollary}\label{cor:subbundles}
Let $X$ be a rational homogeneous variety equal to one of the following $$\DA_n(r),\DB_n(r), \DC_n(r),\mbox{ or }\DD_n(r),$$ for some $r$. The universal quotient bundle $\cQ$ has a nontrivial subbundle if and only if $X=\DA_n(n),\DD_n(n)$, or $\DD_n(n-1)$. 
\end{corollary}

\begin{remark}
Furthermore, a description of the nontrivial subbundles $\cQ'\subset\cQ$ in each case follows from the description of the corresponding nestings, which we will present in section \ref{sec:special}. For instance, we will see that the rank of any such $\cQ'$ is $n-1$ in the case of $\DA_n(n)$, and $1$ in the case of  $\DD_n(n),\DD_n(n-1)$.
\end{remark}

\subsection{Rational homogeneous bundles on ${\bm \P^1}$}\label{ssec:groth}

Rational homogeneous bundles over the projective line $\P^1$ are completely described by Grothendieck's Theorem, which implies that these bundles can be described by some discrete data. We summarize here how to describe them by means of tagged Dynkin diagrams, referring to \cite[Section 3]{OSW} for details and proofs.

With the notation introduced in the previous sections, given a flag bundle $\pi:Y \to \P^1$ and denoting by $H\subset G$ a maximal torus, Grothendieck's Theorem asserts that the natural map $\HH^1(\P^1,H)\to \HH^1(\P^1,G)$ is surjective, so the cocycle in $ \HH^1(\P^1,G)$ determining $Y$ is the image of an element $\theta \in \HH^1(\P^1,H)$.
Moreover, considering the exponential map from $\fh$ to $H$, whose kernel is $\Na(H)$, since $\HH^1(\P^1,\fh)=0$, we have:
$$
\HH^1(\P^1,H)\cong \HH^2(\P^1,\Na(H))\cong\Na(H),
$$
so we can think of $\theta$ as a co-character of $H$.
 
The choice of a Borel subgroup $B \supset H$ determines a cocycle in $\HH^1(\P^1,B)$ which is the image of $\theta$, and thus a section $\Gamma_0$ of $\pi:Y\to \P^1$.
 The section $\Gamma_0$ is a minimal section of $\pi$, i.e., a section  that cannot be deformed in $Y$ with a point fixed, for a choice of $B$ corresponding to the choice of a base $\Delta=\{\alpha_i,\,\,i=1,\dots,n\}$ of positive simple roots of $G$ such that  $d_i:=-\theta(\alpha_i)\geq 0$, for every $i$. The choice of $B$ may not be unique, but one may show that the numerical class of minimal sections is unique.  

The principal bundle given by $\theta$ (and so the corresponding rational homogeneous bundles) is then determined by the $n$-tuple $\delta(\theta)=(d_1, \dots, d_n)$, which we call {\em tag of the bundle}. We usually represent it by the Dynkin diagram $\cD$ labeled with the integer $d_i$ on the node corresponding to $\alpha_i$, for every $i$.
The tag can be given a geometric interpretation in terms of intersection numbers with $\Gamma_0$, as follows: 

\begin{proposition}\label{prop:tag}
Let $\pi:Y\to \P^1$ be a $\cD(D)$-bundle associated with a cocycle $\theta \in \HH^1(\P^1,H)$, and let $\Gamma_0$ be a minimal section of $\pi$ over $\P^1$. 
The tag  $\delta(\theta)=(d_1, \dots, d_n)$ is given by the intersection numbers $K_i \cdot\Gamma_0$ where the $K_i$'s ($i = 1,\dots, n$) are the relative canonical divisors of the elementary contractions of $Y$ over $\P^1$.
\end{proposition}

\subsection{Some lemmas on the Chern classes of nef vector bundles}\label{ssec:lemmata}

We conclude Section \ref{sec:prelim} by stating and proving some technical results on the cohomology of nef vector bundles on varieties of Picard number one, that we will use in the proof of Theorem \ref{thm:main}. Given a vector bundle $\cE$ on a projective variety $X$, we will denote by $\P(\cE)$ its projectivization, defined as $$\P(\cE):=\Proj\left(\bigoplus_{d>0}S^d\cE\right),$$ cf. \cite[II, Section 7]{Ha}, so that the push-forward to $X$ of the tautological line bundle $\cO_{\P(\cE)}(1)$ on $\P(\cE)$ is $\cE$.  Along this section  $X$ will denote a  projective variety such that $\Pic(X)\simeq \Z\langle H\rangle$,  
with $H$ ample.  Let us consider the following properties for a vector bundle $\cE$ over $X$ of rank $\rk(\cE) \le \dim X$:
\begin{equation}
\tag{$*$} \forall i=1, \dots, \rk(\cE)  \textrm{~there exists ~} E_i \in \Q\textrm{~such that~} c_i(\cE) = E_iH^i.
\end{equation}\vspace{-0.6cm}
\begin{equation}\tag{$**$} \forall i=1, \dots, \rk(\cE)  \textrm{~there exists~} E_i \in \Z \textrm{~such that~} c_i(\cE) = E_iH^i.
\end{equation}
By convention, for a vector bundle $\cE$ satisfying $(*)$ or $(**)$ as above, we set $E_0:=1$, and $E_i:=0$ for $i\in \Z\setminus [0,\rk\cE]$, and associate with it the following polynomial:
$$P_\cE(t) := \sum_{i=0}^{\rk \cE} E_it^i,$$
 which is a numerical version of the Chern polynomial of $\cE$.
Moreover, given a sequence of integers
$$\lambda=(\lambda_1,\dots, \lambda_t),\quad\lambda_1\geq \dots\geq \lambda_t> 0,\quad \sum_{i=1}^t\lambda_i=k,$$ 
with $t\leq k\leq \dim(X)$, we define the following rational (resp. integral) numbers: 
$$
S_{\lambda}:=\det\begin{pmatrix}E_{\lambda_1}&E_{\lambda_1+1}&\ldots&E_{\lambda_1+t-1}\\
E_{\lambda_2-1}&E_{\lambda_2}&\ldots&E_{\lambda_2+t-2}\\
\vdots&\vdots&\ddots&\vdots\\
E_{\lambda_t-t+1}&E_{\lambda_t-t+2}&\ldots&E_{\lambda_t}
\end{pmatrix}
$$
Note that multiplying $S_{\lambda}$ by the positive integer $H^{\dim(X)}$ we obtain the intersection number $\Delta_\lambda(\cE)\cdot H^k$, where $\Delta_\lambda(\cE)$ denotes the Schur polynomial associated with $\lambda$ evaluated at $\cE$ (see, for instance, \cite[Example~12.1.7]{Fult}). In the particular case in which $\cE$ is nef (that is, if the tautological line bundle $\cO_{\P(\cE)}(1)$ is nef on the projectivization $\P(\cE)$), one gets that $S_\lambda\geq 0$, for every $\lambda$, by \cite[Theorem~2.5]{DPS}; for our purposes, we will make use the following inequalities:\begin{subequations}
  \begin{alignat}{3}
    &S_{(j)}&=&\quad E_j\geq 0, &  0\leq j,\\
    &S_{(j,1)}&= &\quad E_jE_1-E_{j+1}\geq 0,&\quad\quad  1\leq j+1 \le \dim X,\label{sub:3a}\\
    &S_{(j,j)}&= &\quad E^2_j -E_{j-1}E_{j+1}\geq 0,&\quad\quad 0\leq 2j \le \dim X,\label{sub:3b}\\
    &S_{(j,j,1)}\,\,&= &\quad E_1S_{(j,j)} -E_{j+1}E_{j}+E_{j+2}E_{j-1}\geq 0,&\qquad\quad  1\leq 2j+1 \le \dim X.\label{sub:3c}
  \end{alignat}
\end{subequations}

\begin{lemma}\label{lem:c1} Let $\cE$ be a nef vector bundle satisfying $(**)$. Let $r$ be the maximum integer such that $E_r \not = 0$. Then $E_j >0$ for every $j=1, \dots, r$.
Moreover, if $s:=\min\{r-1, \lfloor\frac{\dim X}{2} \rfloor\} >0$, then one of the following holds: 
\begin{enumerate}
\item $E_i=1$ for every $i=1, \dots, s+1$; 
\item $E_i \ge 2$ for every $i=1, \dots, s$. 
\end{enumerate}
\end{lemma}

\begin{proof}
From the nonnegativity of $S_{(j,1)}= E_jE_1-E_{j+1}$ and the hypothesis $E_r\neq 0$, we get that $E_j >0$ for every $j=1, \dots, r$.

Assume that there exists an index $1 \le j \le s$ such that $E_j=1$. Since $s\leq r-1$ by definition, then $j+1\leq r$ and $j-1\geq 1$, so, both $E_{j-1}$ and $E_{j+1}$ are not $0$. Then the nonnegativity of $S_{(j,j)}= \, E^2_j -E_{j-1}E_{j+1}$, which we can apply because $j\leq s\leq \lfloor\frac{\dim X}{2} \rfloor$, implies that  $E_{j-1}=E_{j+1}=1$.
\end{proof}

\begin{lemma}\label{lem:c3} Let $\cE$ be a nef vector bundle satisfying $(**)$. Let $r$ be the maximum integer such that $E_r \not = 0$, and
assume that $1 < r \le \lfloor \frac{\dim X+1}{2} \rfloor$.
Then $E_i \ge 2$ for every $i=1, \dots, r-1$. If moreover $E_i = 2$ for every $i=1, \dots, r-1$, and $E_r=1$,  
then $r \le 3$.
\end{lemma}

\begin{proof}
By hypothesis $s=\min\{r-1, \lfloor\frac{\dim X}{2} \rfloor\}=r-1$, hence by Lemma \ref{lem:c1} it is enough to prove that $E_{r-1} >1$.
If else $E_{r-1} =1$, by inequality (\ref{sub:3b}) we  have $S_{(r-1,r-1)}=0$; therefore,
$S_{(r-1,r-1,1)}$ is equal to $-E_{r}E_{r-1}$, which is strictly negative, contradicting (\ref{sub:3c}). 

The assumption that $E_i = 2$ for every $i=1, \dots, r-1$ implies that  $S_{(j,j)}=0$ for $2 \le j \le r-2$. This condition is empty for $r\leq 3$, while if $r\ge 4$ we at least have 
$S_{(r-2,r-2)}=0$. Then, since we are also assuming that $E_r=1$, in the case $r\geq 4$ we would have:
$$S_{(r-2,r-2,1)} = E_1S_{(r-2,r-2)}-(E_{r-1}E_{r-2}-E_{r}E_{r-3}) =-2,$$
a contradiction.
\end{proof}

\begin{lemma}\label{lem:prodpol}
Let $\cE$ and $\cF$ be two nef vector bundles on $X$ satisfying $(*)$ and such that 
$$P_{\cE}(t)P_{\cF}(-t)= 1 -t^{k}$$ 
for some  integer $k$ such that $k \le \dim X+1$.  Then one of the following happens:
\begin{enumerate}
\item $k=6$, $P_{\cE}(t)=P_{\cF}(t)=1+2t+2t^2+t^3$; 
\item $P_{\cE}(t)=\sum_{i=0}^{k-1} t^i$, $P_{\cF}(t)=1+t$;
\item $P_{\cE}(t)=1+t$, $P_{\cF}(t)=\sum_{i=0}^{k-1} t^i$, $k$ even.
\end{enumerate}
\end{lemma}

\begin{proof}
Write $$P_{\cE}(t)= \sum_{i=0}^e E_it^i, \qquad P_{\cF}(t)=\sum_{i=0}^f F_it^i,$$
with $E_e,F_f\neq 0$. Since the product $P_\cE(t)P_\cF(-t)=1-t^{k}$ is primitive, Gauss' lemma tells us that the coefficients of $P_\cE(t)$ and $P_\cF(-t)$
are integers, and the bundles $\cE,\cF$ satisfy condition $(**)$. Moreover, since $(-1)^{f}E_{e}F_{f} =-1$ and $E_e,F_f>0$ by Lemma \ref{lem:c3}, we deduce that $f$ is odd, and that $E_e=F_f=1$. If $f=1$ then clearly we are in case (2); assume from now on that $f \ge 3$.

We claim that $k$ must be even: in fact, if this is not the case, evaluating at $t=-1$ we get
$P_\cE(-1)P_\cF(1)=2$, hence $\sum_{i=0}^{f} F_i \le 2$, which, by the first part of Lemma \ref{lem:c1}, implies that $f \le 1$, against our assumption that $f \ge 3$. Now, changing $t$ to $-t$ we get $P_{\cE}(-t)P_{\cF}(t)= 1 -t^{k}$, and our previous argument tells us that, either we are in case (3), or $e\geq 3$. 

Summing up, we are left with the case in which $k=2k'$, $k'\in\Z$, $e,f\geq 3$. Note first that in this case we have $P_{\cE}(-t)P_{\cF}(t)=P_{\cE}(t)P_{\cF}(-t)= 1 -t^{k}$ so we can assume, without loss of generality, that $e\geq f$. In particular, we may apply Lemma \ref{lem:c3} to $\cF$, obtaining that
$F_1 \ge 2$. Since $E_1-F_1=0$ this implies that $E_1\geq 2$, and Lemma \ref{lem:c1} tells us that $E_i \ge 2$ for every $i=1, \dots, k'-1$  (note that $k'-1<k/2\leq e$ and that, by hypothesis $k'-1=(k/2)-1\leq\lfloor (\dim X+1)/2\rfloor-1$, hence $k'-1\leq \min\{e-1,\lfloor \dim X/2\rfloor\}$).

Evaluating at $t=1$ we get $P_\cE(1)P_\cF(-1)=0$; since $P_\cE(1) > 0$, then $P_\cF(-1)=0$ and we may factor $P_\cF(-t)=(1-t)G(t)$, with $G(t)\in\Z[t]$.
We can thus write
\begin{equation}\label{eq:1}
P_\cE(t)G(t)= \sum_{i=0}^{2k'-1} t^{i}
\end{equation}
In particular, evaluating at $t=1$ we deduce that $\sum_i E_i$ divides $2k'$. Since
$$\sum_{i=0}^{e} E_i \ge 1 +2(k'-1) + \sum_{i=k'}^{e}E_i \ge 2k',$$ we must have $\sum E_i = 2k'$ and $e=k'$, which force $E_i=2$ for every $i=1, \dots e-1$. We conclude by the second part of Lemma \ref{lem:c3}.
\end{proof}



\section{Nesting rational homogeneous varieties}\label{sec:nesting}
In this section we recall the definition of nesting of rational homogeneous varieties, and present some examples of nestings that are constructed upon diagram foldings. The last part of the section is devoted to the behaviour of nestings in bundles, that we will later apply in the case of bundles over the projective line $\P^1$. 

\begin{definition}
Let $G$ be a semisimple group, with Lie algebra $\fg$, and Dynkin diagram $\cD$. Given two disjoint nonempty subsets of nodes of $\cD$, $I,J\subset D$, {\em a nesting of type $(\cD,I,J)$} is a section $\sigma:\cD(I)\to\cD(I\cup J)$ of the natural map $\pi_{I\cup J,I}:\cD(I\cup J)\to \cD(I)$. \end{definition}

We may provide the set of  nestings of type $(\cD,I,J)$ with a structure of scheme, by considering the schemes $\Hom(\cD(I),\cD(I\cup J))$, $\Hom(\cD(I),\cD(I))$ (see \cite[I~1.9]{Kollar}), the natural morphism 
$$
\Hom(\cD(I),\cD(I\cup J))\to \Hom(\cD(I),\cD(I))
$$
given by the composition with $\pi_{I\cup J,I}$, and defining  $\Nest(\cD,I,J)$ as the scheme-theoretical inverse image of the identity morphism. As such, it comes with a universal family $\Nest(\cD,I,J)\times \cD(I)\to\cD(I\cup J)$, restriction of the universal family of $\Hom(\cD(I),\cD(I\cup J))$, that represents the functor (which we also denote by $\Nest(\cD,I,J)$) associating to every $\C$-scheme $S$ the set of morphisms $S\times \cD(I)\to \cD(I\cup J)$ satisfying that their composition with $\pi_{I\cup J,I}$ equals the second projection $S\times \cD(I)\to \cD(I)$.

%
%

\begin{remark}\label{rem:action}
Assume now that $\Nest(\cD,I,J)\neq\emptyset$. Since $G$ is the adjoint group of type $\cD$, its action on $\cD(I)$ and $\cD(I\cup J)$ induces an algebraic action on $\Nest(\cD,I,J)$, given by the conjugation with $g \in G$:
$$
\xymatrix@C=90pt{\cD(I)\ar[r]^{\sigma\in \Nest(\cD,I,J)}&\cD(I\cup J)\ar[d]^{g^{-1}}\\\cD(I)\ar[r]_{g^{-1}\sigma g\,\in \Nest(\cD,I,J)}\ar[u]^{ g}&\cD(I\cup J)}
$$
In this way, the isotropy subgroup of $\sigma\in \Nest(\cD,I,J)$ in  $G$ is precisely 
$$
G_\sigma:=\left\{g\in G|\,\, g\sigma=\sigma g\right\}.
$$
The action of $G$ on $\Nest(\cD,I,J)$ is defined scheme-theoretically, by means of the universal property of $\Nest(\cD,I,J)$, in the obvious way.  
\end{remark}

\subsection{Examples: nestings and foldings}\label{ssec:folds}

In the case in which the Dynkin diagram $\cD$ is disconnected one may easily construct examples of nestings, in the following way:
\begin{example}
Given two semisimple groups $G,G'$ with Dynkin diagrams $\cD,\cD'$, and given a morphism $f:\cD(I)\to \cD'(I')$, the graph of $f$ defines a nesting $(\cD\sqcup\cD',I,I')$. 
\end{example}

Finding examples of nestings in the case in which $\cD$ is connected is more involved. The obvious examples are related to the existence of {\em foldings} of $\cD$, that we discuss in detail hereunder.

We consider the following pairs $(G',G)$ of semisimple algebraic groups:
\begin{subequations}
  \begin{alignat}{3}
    &(\PSP(2n),\PGL(2n)),& &\qquad n\geq 2, \label{sub:gpa} \\
    &(\PSO(2n-1),\PSO(2n)),& &\qquad n\geq 4,\label{sub:gpd}\\
    &(\DF_4,\DE_{6}^{\ad}),& &\label{sub:gpe}\\
    &(\DG_2,\PSO(8)),& &\label{sub:gpd4}\\
    &(\DG_2,\PSO(7)),& &\label{sub:gpb3}
  \end{alignat}
\end{subequations}
where $\DE_{6}^{\ad}$, $\DF_4$, $\DG_2$ represent the adjoint  algebraic groups of the complex simple  Lie algebras of type $\DF_4$, $\DG_2$,  and $\DE_{6}$, respectively. It is well known that in each case we have an inclusion $G'\subset G$ that can be expressed in terms of a map from the root system of $G$ to the root system of $G'$, in the following sense: we may choose a Cartan and a Borel subgroup of $G$, $H\subset B\subset G$, such that $H':=H\cap G'\subset B':=B\cap G'\subset G'$ are Cartan and Borel subgroups of $G'$, respectively. Then the induced map $p:\Mo(H)\to \Mo(H')$ among the corresponding lattices of characters sends the base of positive simple roots of $G$ determined by $B$ to the base of positive simple roots of $G'$ determined by $B'$. In other words, $p$ can be determined by a map among the sets of nodes $D$, $D'$ of the corresponding Dynkin diagrams $\cD$, $\cD'$. Graphically, this operation is called {\em diagram folding}: 

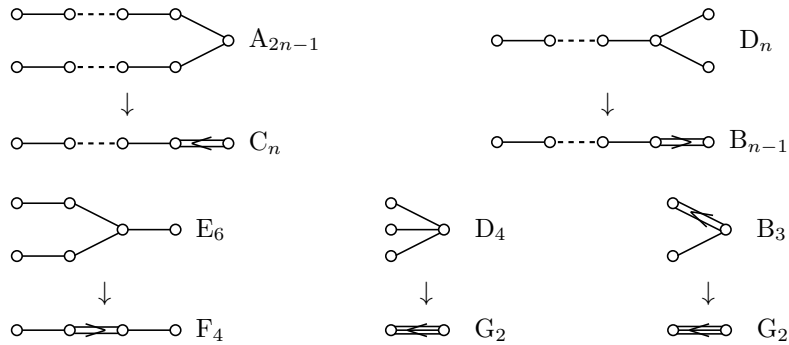
\renewcommand\arraycolsep{30pt}
$$
\begin{array}{ll}
\ifx\du\undefined
  \newlength{\du}
\fi
\setlength{\du}{2\unitlength}
\begin{tikzpicture}
\pgftransformxscale{1.000000}
\pgftransformyscale{1.000000}
\definecolor{dialinecolor}{rgb}{0.000000, 0.000000, 0.000000} 
\pgfsetstrokecolor{dialinecolor}
\definecolor{dialinecolor}{rgb}{0.000000, 0.000000, 0.000000} 
\pgfsetfillcolor{dialinecolor}
\pgfsetlinewidth{0.300000\du}
\pgfsetdash{}{0pt}
\pgfsetdash{}{0pt}
\pgfpathellipse{\pgfpoint{4\du}{-5\du}}{\pgfpoint{1\du}{0\du}}{\pgfpoint{0\du}{1\du}}
\pgfusepath{stroke}
\node at (4\du,-5\du){};
\pgfpathellipse{\pgfpoint{4\du}{5\du}}{\pgfpoint{1\du}{0\du}}{\pgfpoint{0\du}{1\du}}
\pgfusepath{stroke}
\node at (4\du,5\du){};
\pgfpathellipse{\pgfpoint{14\du}{-5\du}}{\pgfpoint{1\du}{0\du}}{\pgfpoint{0\du}{1\du}}
\pgfusepath{stroke}
\node at (14\du,-5\du){};
\pgfpathellipse{\pgfpoint{14\du}{5\du}}{\pgfpoint{1\du}{0\du}}{\pgfpoint{0\du}{1\du}}
\pgfusepath{stroke}
\node at (14\du,5\du){};
\pgfpathellipse{\pgfpoint{24\du}{-5\du}}{\pgfpoint{1\du}{0\du}}{\pgfpoint{0\du}{1\du}}
\pgfusepath{stroke}
\node at (24\du,-5\du){};
\pgfpathellipse{\pgfpoint{24\du}{5\du}}{\pgfpoint{1\du}{0\du}}{\pgfpoint{0\du}{1\du}}
\pgfusepath{stroke}
\node at (24\du,5\du){};
\pgfpathellipse{\pgfpoint{34\du}{5\du}}{\pgfpoint{1\du}{0\du}}{\pgfpoint{0\du}{1\du}}
\pgfusepath{stroke}
\node at (34\du,5\du){};
\pgfpathellipse{\pgfpoint{34\du}{-5\du}}{\pgfpoint{1\du}{0\du}}{\pgfpoint{0\du}{1\du}}
\pgfusepath{stroke}
\node at (34\du,-5\du){};
\pgfpathellipse{\pgfpoint{44\du}{0\du}}{\pgfpoint{1\du}{0\du}}{\pgfpoint{0\du}{1\du}}
\pgfusepath{stroke}
\node at (44\du,0\du){};
\pgfsetlinewidth{0.300000\du}
\pgfsetdash{}{0pt}
\pgfsetdash{}{0pt}
\pgfsetbuttcap
{\draw (5\du,5\du)--(13\du,5\du);}
{\draw (5\du,-5\du)--(13\du,-5\du);}
{\draw (25\du,5\du)--(33\du,5\du);}
{\draw (25\du,-5\du)--(33\du,-5\du);}
{\draw (34.85\du,4.9\du)--(43.35\du,0.6\du);}
{\draw (34.85\du,-4.9\du)--(43.35\du,-0.6\du);}
\pgfsetlinewidth{0.400000\du}
\pgfsetdash{{1.000000\du}{1.000000\du}}{0\du}
\pgfsetdash{{1.000000\du}{1.000000\du}}{0\du}
\pgfsetbuttcap
{\draw (15.5\du,-5\du)--(22.8\du,-5\du);}
{\draw (15.5\du,5\du)--(22.8\du,5\du);}
\node[anchor=west] at (46\du,0\du){${\rm A}_{2n-1}$};
\end{tikzpicture} 
&
\ifx\du\undefined
  \newlength{\du}
\fi
\setlength{\du}{2\unitlength}
\begin{tikzpicture}
\pgftransformxscale{1.000000}
\pgftransformyscale{1.000000}
\definecolor{dialinecolor}{rgb}{0.000000, 0.000000, 0.000000} 
\pgfsetstrokecolor{dialinecolor}
\definecolor{dialinecolor}{rgb}{0.000000, 0.000000, 0.000000} 
\pgfsetfillcolor{dialinecolor}
\pgfsetlinewidth{0.300000\du}
\pgfsetdash{}{0pt}
\pgfsetdash{}{0pt}
\pgfpathellipse{\pgfpoint{-6\du}{0\du}}{\pgfpoint{1\du}{0\du}}{\pgfpoint{0\du}{1\du}}
\pgfusepath{stroke}
\node at (-6\du,0\du){};
\pgfpathellipse{\pgfpoint{4\du}{0\du}}{\pgfpoint{1\du}{0\du}}{\pgfpoint{0\du}{1\du}}
\pgfusepath{stroke}
\node at (4\du,0\du){};
\pgfpathellipse{\pgfpoint{14\du}{0\du}}{\pgfpoint{1\du}{0\du}}{\pgfpoint{0\du}{1\du}}
\pgfusepath{stroke}
\node at (14\du,0\du){};
\pgfpathellipse{\pgfpoint{24\du}{0\du}}{\pgfpoint{1\du}{0\du}}{\pgfpoint{0\du}{1\du}}
\pgfusepath{stroke}
\node at (24\du,0\du){};
\pgfpathellipse{\pgfpoint{34\du}{5\du}}{\pgfpoint{1\du}{0\du}}{\pgfpoint{0\du}{1\du}}
\pgfusepath{stroke}
\node at (34\du,5\du){};
\pgfpathellipse{\pgfpoint{34\du}{-5\du}}{\pgfpoint{1\du}{0\du}}{\pgfpoint{0\du}{1\du}}
\pgfusepath{stroke}
\node at (34\du,-5\du){};
\pgfsetlinewidth{0.300000\du}
\pgfsetdash{}{0pt}
\pgfsetdash{}{0pt}
\pgfsetbuttcap
{\draw (-5\du,0\du)--(3\du,0\du);}
{\draw (15\du,0\du)--(23\du,0\du);}
{\draw (24.65\du,0.6\du)--(33.1\du,4.9\du);}
{\draw (24.65\du,-0.6\du)--(33.1\du,-4.9\du);}
\pgfsetlinewidth{0.400000\du}
\pgfsetdash{{1.000000\du}{1.000000\du}}{0\du}
\pgfsetdash{{1.000000\du}{1.000000\du}}{0\du}
\pgfsetbuttcap
{\draw (5.5\du,0\du)--(12.8\du,0\du);}
\node[anchor=west] at (38\du,0\du){${\rm D}_n$};
\end{tikzpicture} 
\\
\hspace{1.5cm}\downarrow&\hspace{1.5cm}\downarrow\\

\ifx\du\undefined
  \newlength{\du}
\fi
\setlength{\du}{2\unitlength}
\begin{tikzpicture}
\pgftransformxscale{1.000000}
\pgftransformyscale{1.000000}
\definecolor{dialinecolor}{rgb}{0.000000, 0.000000, 0.000000} 
\pgfsetstrokecolor{dialinecolor}
\definecolor{dialinecolor}{rgb}{0.000000, 0.000000, 0.000000} 
\pgfsetfillcolor{dialinecolor}
\pgfsetlinewidth{0.300000\du}
\pgfsetdash{}{0pt}
\pgfsetdash{}{0pt}
\pgfpathellipse{\pgfpoint{4\du}{0\du}}{\pgfpoint{1\du}{0\du}}{\pgfpoint{0\du}{1\du}}
\pgfusepath{stroke}
\node at (4\du,0\du){};
\pgfpathellipse{\pgfpoint{14\du}{0\du}}{\pgfpoint{1\du}{0\du}}{\pgfpoint{0\du}{1\du}}
\pgfusepath{stroke}
\node at (14\du,0\du){};
\pgfpathellipse{\pgfpoint{24\du}{0\du}}{\pgfpoint{1\du}{0\du}}{\pgfpoint{0\du}{1\du}}
\pgfusepath{stroke}
\node at (24\du,0\du){};
\pgfpathellipse{\pgfpoint{34\du}{0\du}}{\pgfpoint{1\du}{0\du}}{\pgfpoint{0\du}{1\du}}
\pgfusepath{stroke}
\node at (34\du,0\du){};
\pgfpathellipse{\pgfpoint{44\du}{0\du}}{\pgfpoint{1\du}{0\du}}{\pgfpoint{0\du}{1\du}}
\pgfusepath{stroke}
\node at (44\du,0\du){};
\pgfsetlinewidth{0.300000\du}
\pgfsetdash{}{0pt}
\pgfsetdash{}{0pt}
\pgfsetbuttcap
{\draw (5\du,0\du)--(13\du,0\du);}
{\draw (25\du,0\du)--(33\du,0\du);}
{\draw (34.85\du,0.6\du)--(43.35\du,0.6\du);}
{\draw (34.85\du,-0.6\du)--(43.35\du,-0.6\du);}
{\pgfsetcornersarced{\pgfpoint{0.300000\du}{0.300000\du}}\definecolor{dialinecolor}{rgb}{0.000000, 0.000000, 0.000000}
\pgfsetstrokecolor{dialinecolor}
\draw (40.8\du,1.2\du)--(37\du,0\du)--(40.8\du,-1.2\du);}
\pgfsetlinewidth{0.400000\du}
\pgfsetdash{{1.000000\du}{1.000000\du}}{0\du}
\pgfsetdash{{1.000000\du}{1.000000\du}}{0\du}
\pgfsetbuttcap
{\draw (15.5\du,0\du)--(22.8\du,0\du);}
\node[anchor=west] at (46\du,0\du){${\rm C}_{n}$};
\end{tikzpicture} 
&
\ifx\du\undefined
  \newlength{\du}
\fi
\setlength{\du}{2\unitlength}
\begin{tikzpicture}
\pgftransformxscale{1.000000}
\pgftransformyscale{1.000000}
\definecolor{dialinecolor}{rgb}{0.000000, 0.000000, 0.000000} 
\pgfsetstrokecolor{dialinecolor}
\definecolor{dialinecolor}{rgb}{0.000000, 0.000000, 0.000000} 
\pgfsetfillcolor{dialinecolor}
\pgfsetlinewidth{0.300000\du}
\pgfsetdash{}{0pt}
\pgfsetdash{}{0pt}
\pgfpathellipse{\pgfpoint{4\du}{0\du}}{\pgfpoint{1\du}{0\du}}{\pgfpoint{0\du}{1\du}}
\pgfusepath{stroke}
\node at (4\du,0\du){};
\pgfpathellipse{\pgfpoint{14\du}{0\du}}{\pgfpoint{1\du}{0\du}}{\pgfpoint{0\du}{1\du}}
\pgfusepath{stroke}
\node at (14\du,0\du){};
\pgfpathellipse{\pgfpoint{24\du}{0\du}}{\pgfpoint{1\du}{0\du}}{\pgfpoint{0\du}{1\du}}
\pgfusepath{stroke}
\node at (24\du,0\du){};
\pgfpathellipse{\pgfpoint{34\du}{0\du}}{\pgfpoint{1\du}{0\du}}{\pgfpoint{0\du}{1\du}}
\pgfusepath{stroke}
\node at (34\du,0\du){};
\pgfpathellipse{\pgfpoint{44\du}{0\du}}{\pgfpoint{1\du}{0\du}}{\pgfpoint{0\du}{1\du}}
\pgfusepath{stroke}
\node at (44\du,0\du){};
\pgfsetlinewidth{0.300000\du}
\pgfsetdash{}{0pt}
\pgfsetdash{}{0pt}
\pgfsetbuttcap
{\draw (5\du,0\du)--(13\du,0\du);}
{\draw (25\du,0\du)--(33\du,0\du);}
{\draw (34.85\du,0.6\du)--(43.35\du,0.6\du);}
{\draw (34.85\du,-0.6\du)--(43.35\du,-0.6\du);}
{\pgfsetcornersarced{\pgfpoint{0.300000\du}{0.300000\du}}\definecolor{dialinecolor}{rgb}{0.000000, 0.000000, 0.000000}
\pgfsetstrokecolor{dialinecolor}
\draw (37\du,1.2\du)--(40.8\du,0\du)--(37\du,-1.2\du);}
\pgfsetlinewidth{0.400000\du}
\pgfsetdash{{1.000000\du}{1.000000\du}}{0\du}
\pgfsetdash{{1.000000\du}{1.000000\du}}{0\du}
\pgfsetbuttcap
{\draw (15.5\du,0\du)--(22.8\du,0\du);}
\node[anchor=west] at (46\du,0\du){${\rm B}_{n-1}$};
\end{tikzpicture} 
\end{array}
$$
\renewcommand\arraycolsep{28pt}
$$
\begin{array}{lll}
\ifx\du\undefined
  \newlength{\du}
\fi
\setlength{\du}{2\unitlength}
\begin{tikzpicture}
\pgftransformxscale{1.000000}
\pgftransformyscale{1.000000}
\definecolor{dialinecolor}{rgb}{0.000000, 0.000000, 0.000000} 
\pgfsetstrokecolor{dialinecolor}
\definecolor{dialinecolor}{rgb}{0.000000, 0.000000, 0.000000} 
\pgfsetfillcolor{dialinecolor}
\pgfsetlinewidth{0.300000\du}
\pgfsetdash{}{0pt}
\pgfsetdash{}{0pt}
\pgfpathellipse{\pgfpoint{24\du}{-5\du}}{\pgfpoint{1\du}{0\du}}{\pgfpoint{0\du}{1\du}}
\pgfusepath{stroke}
\node at (24\du,-5\du){};
\pgfpathellipse{\pgfpoint{24\du}{5\du}}{\pgfpoint{1\du}{0\du}}{\pgfpoint{0\du}{1\du}}
\pgfusepath{stroke}
\node at (24\du,5\du){};
\pgfpathellipse{\pgfpoint{34\du}{5\du}}{\pgfpoint{1\du}{0\du}}{\pgfpoint{0\du}{1\du}}
\pgfusepath{stroke}
\node at (34\du,5\du){};
\pgfpathellipse{\pgfpoint{34\du}{-5\du}}{\pgfpoint{1\du}{0\du}}{\pgfpoint{0\du}{1\du}}
\pgfusepath{stroke}
\node at (34\du,-5\du){};
\pgfpathellipse{\pgfpoint{44\du}{0\du}}{\pgfpoint{1\du}{0\du}}{\pgfpoint{0\du}{1\du}}
\pgfusepath{stroke}
\node at (44\du,0\du){};
\pgfpathellipse{\pgfpoint{54\du}{0\du}}{\pgfpoint{1\du}{0\du}}{\pgfpoint{0\du}{1\du}}
\pgfusepath{stroke}
\node at (54\du,0\du){};
\pgfsetlinewidth{0.300000\du}
\pgfsetdash{}{0pt}
\pgfsetdash{}{0pt}
\pgfsetbuttcap
{\draw (25\du,5\du)--(33\du,5\du);}
{\draw (25\du,-5\du)--(33\du,-5\du);}
{\draw (34.85\du,4.9\du)--(43.35\du,0.6\du);}
{\draw (34.85\du,-4.9\du)--(43.35\du,-0.6\du);}
{\draw (45\du,0\du)--(53\du,0\du);}
\node[anchor=west] at (56\du,0\du){${\rm E}_{6}$};
\end{tikzpicture} 
&
\ifx\du\undefined
  \newlength{\du}
\fi
\setlength{\du}{2\unitlength}
\begin{tikzpicture}
\pgftransformxscale{1.000000}
\pgftransformyscale{1.000000}
\definecolor{dialinecolor}{rgb}{0.000000, 0.000000, 0.000000} 
\pgfsetstrokecolor{dialinecolor}
\definecolor{dialinecolor}{rgb}{0.000000, 0.000000, 0.000000} 
\pgfsetfillcolor{dialinecolor}
\pgfsetlinewidth{0.300000\du}
\pgfsetdash{}{0pt}
\pgfsetdash{}{0pt}
\pgfpathellipse{\pgfpoint{34\du}{5\du}}{\pgfpoint{1\du}{0\du}}{\pgfpoint{0\du}{1\du}}
\pgfusepath{stroke}
\node at (34\du,5\du){};
\pgfpathellipse{\pgfpoint{34\du}{-5\du}}{\pgfpoint{1\du}{0\du}}{\pgfpoint{0\du}{1\du}}
\pgfusepath{stroke}
\node at (34\du,-5\du){};
\pgfpathellipse{\pgfpoint{44\du}{0\du}}{\pgfpoint{1\du}{0\du}}{\pgfpoint{0\du}{1\du}}
\pgfusepath{stroke}
\node at (44\du,0\du){};
\pgfpathellipse{\pgfpoint{34\du}{0\du}}{\pgfpoint{1\du}{0\du}}{\pgfpoint{0\du}{1\du}}
\pgfusepath{stroke}
\node at (34\du,0\du){};
\pgfsetlinewidth{0.300000\du}
\pgfsetdash{}{0pt}
\pgfsetdash{}{0pt}
\pgfsetbuttcap
{\draw (35\du,0\du)--(43\du,0\du);}
{\draw (34.85\du,4.9\du)--(43.35\du,0.6\du);}
{\draw (34.85\du,-4.9\du)--(43.35\du,-0.6\du);}
\node[anchor=west] at (48\du,0\du){${\rm D}_{4}$};
\end{tikzpicture} 
&
\ifx\du\undefined
  \newlength{\du}
\fi
\setlength{\du}{2\unitlength}
\begin{tikzpicture}
\pgftransformxscale{1.000000}
\pgftransformyscale{1.000000}
\definecolor{dialinecolor}{rgb}{0.000000, 0.000000, 0.000000} 
\pgfsetstrokecolor{dialinecolor}
\definecolor{dialinecolor}{rgb}{0.000000, 0.000000, 0.000000} 
\pgfsetfillcolor{dialinecolor}
\pgfsetlinewidth{0.300000\du}
\pgfsetdash{}{0pt}
\pgfsetdash{}{0pt}
\pgfpathellipse{\pgfpoint{34\du}{5\du}}{\pgfpoint{1\du}{0\du}}{\pgfpoint{0\du}{1\du}}
\pgfusepath{stroke}
\node at (34\du,5\du){};
\pgfpathellipse{\pgfpoint{34\du}{-5\du}}{\pgfpoint{1\du}{0\du}}{\pgfpoint{0\du}{1\du}}
\pgfusepath{stroke}
\node at (34\du,-5\du){};
\pgfpathellipse{\pgfpoint{44\du}{0\du}}{\pgfpoint{1\du}{0\du}}{\pgfpoint{0\du}{1\du}}
\pgfusepath{stroke}
\node at (44\du,0\du){};
\pgfsetlinewidth{0.300000\du}
\pgfsetdash{}{0pt}
\pgfsetdash{}{0pt}
\pgfsetbuttcap
{\draw (34.95\du,5.3\du)--(43.65\du,0.95\du);}
{\draw (34.45\du,4.1\du)--(43.15\du,-0.25\du);}
{\draw (34.85\du,-4.9\du)--(43.35\du,-0.6\du);}
{\pgfsetcornersarced{\pgfpoint{0.300000\du}{0.300000\du}}\definecolor{dialinecolor}{rgb}{0.000000, 0.000000, 0.000000}
\pgfsetstrokecolor{dialinecolor}
\draw (40.4\du,0.5\du)--(37.3\du,3.4\du)--(41.6\du,2.7\du);}
\node[anchor=west] at (48\du,0\du){${\rm B}_{3}$};
\end{tikzpicture} 
\\
\hspace{1.2cm}\downarrow&\hspace{0.5cm}\downarrow&\hspace{0.5cm}\downarrow\\
\ifx\du\undefined
  \newlength{\du}
\fi
\setlength{\du}{2\unitlength}
\begin{tikzpicture}
\pgftransformxscale{1.000000}
\pgftransformyscale{1.000000}
\definecolor{dialinecolor}{rgb}{0.000000, 0.000000, 0.000000} 
\pgfsetstrokecolor{dialinecolor}
\definecolor{dialinecolor}{rgb}{0.000000, 0.000000, 0.000000} 
\pgfsetfillcolor{dialinecolor}
\pgfsetlinewidth{0.300000\du}
\pgfsetdash{}{0pt}
\pgfsetdash{}{0pt}
\pgfpathellipse{\pgfpoint{4\du}{0\du}}{\pgfpoint{1\du}{0\du}}{\pgfpoint{0\du}{1\du}}
\pgfusepath{stroke}
\node at (4\du,0\du){};
\pgfpathellipse{\pgfpoint{14\du}{0\du}}{\pgfpoint{1\du}{0\du}}{\pgfpoint{0\du}{1\du}}
\pgfusepath{stroke}
\node at (14\du,0\du){};
\pgfpathellipse{\pgfpoint{24\du}{0\du}}{\pgfpoint{1\du}{0\du}}{\pgfpoint{0\du}{1\du}}
\pgfusepath{stroke}
\node at (24\du,0\du){};
\pgfpathellipse{\pgfpoint{34\du}{0\du}}{\pgfpoint{1\du}{0\du}}{\pgfpoint{0\du}{1\du}}
\pgfusepath{stroke}
\node at (34\du,0\du){};
\pgfsetlinewidth{0.300000\du}
\pgfsetdash{}{0pt}
\pgfsetdash{}{0pt}
\pgfsetbuttcap
{\draw (5\du,0\du)--(13\du,0\du);}
{\draw (25\du,0\du)--(33\du,0\du);}
{\draw (14.85\du,0.6\du)--(23.35\du,0.6\du);}
{\draw (14.85\du,-0.6\du)--(23.35\du,-0.6\du);}
{\pgfsetcornersarced{\pgfpoint{0.300000\du}{0.300000\du}}\definecolor{dialinecolor}{rgb}{0.000000, 0.000000, 0.000000}
\pgfsetstrokecolor{dialinecolor}
\draw (17\du,1.2\du)--(20.8\du,0\du)--(17\du,-1.2\du);}
\node[anchor=west] at (36\du,0\du){${\rm F}_{4}$};
\end{tikzpicture} 
&
\ifx\du\undefined
  \newlength{\du}
\fi
\setlength{\du}{2\unitlength}
\begin{tikzpicture}
\pgftransformxscale{1.000000}
\pgftransformyscale{1.000000}
\definecolor{dialinecolor}{rgb}{0.000000, 0.000000, 0.000000} 
\pgfsetstrokecolor{dialinecolor}
\definecolor{dialinecolor}{rgb}{0.000000, 0.000000, 0.000000} 
\pgfsetfillcolor{dialinecolor}
\pgfsetlinewidth{0.300000\du}
\pgfsetdash{}{0pt}
\pgfsetdash{}{0pt}
\pgfpathellipse{\pgfpoint{-6\du}{0\du}}{\pgfpoint{1\du}{0\du}}{\pgfpoint{0\du}{1\du}}
\pgfusepath{stroke}
\node at (-6\du,0\du){};
\pgfpathellipse{\pgfpoint{4\du}{0\du}}{\pgfpoint{1\du}{0\du}}{\pgfpoint{0\du}{1\du}}
\pgfusepath{stroke}
\node at (4\du,0\du){};
\pgfsetlinewidth{0.300000\du}
\pgfsetdash{}{0pt}
\pgfsetdash{}{0pt}
\pgfsetbuttcap
{\draw (-5\du,0\du)--(3\du,0\du);}
{\draw (-5.35\du,0.7\du)--(3.35\du,0.7\du);}
{\draw (-5.35\du,-0.7\du)--(3.35\du,-0.7\du);}
{\pgfsetcornersarced{\pgfpoint{0.300000\du}{0.300000\du}}\definecolor{dialinecolor}{rgb}{0.000000, 0.000000, 0.000000}
\pgfsetstrokecolor{dialinecolor}
\draw (0.8\du,-1.2\du)--(-3\du,0\du)--(0.8\du,1.2\du);}
\node[anchor=west] at (8\du,0\du){${\rm G}_2$};
\end{tikzpicture} 
&
\ifx\du\undefined
  \newlength{\du}
\fi
\setlength{\du}{2\unitlength}
\begin{tikzpicture}
\pgftransformxscale{1.000000}
\pgftransformyscale{1.000000}
\definecolor{dialinecolor}{rgb}{0.000000, 0.000000, 0.000000} 
\pgfsetstrokecolor{dialinecolor}
\definecolor{dialinecolor}{rgb}{0.000000, 0.000000, 0.000000} 
\pgfsetfillcolor{dialinecolor}
\pgfsetlinewidth{0.300000\du}
\pgfsetdash{}{0pt}
\pgfsetdash{}{0pt}
\pgfpathellipse{\pgfpoint{-6\du}{0\du}}{\pgfpoint{1\du}{0\du}}{\pgfpoint{0\du}{1\du}}
\pgfusepath{stroke}
\node at (-6\du,0\du){};
\pgfpathellipse{\pgfpoint{4\du}{0\du}}{\pgfpoint{1\du}{0\du}}{\pgfpoint{0\du}{1\du}}
\pgfusepath{stroke}
\node at (4\du,0\du){};
\pgfsetlinewidth{0.300000\du}
\pgfsetdash{}{0pt}
\pgfsetdash{}{0pt}
\pgfsetbuttcap
{\draw (-5\du,0\du)--(3\du,0\du);}
{\draw (-5.35\du,0.7\du)--(3.35\du,0.7\du);}
{\draw (-5.35\du,-0.7\du)--(3.35\du,-0.7\du);}
{\pgfsetcornersarced{\pgfpoint{0.300000\du}{0.300000\du}}\definecolor{dialinecolor}{rgb}{0.000000, 0.000000, 0.000000}
\pgfsetstrokecolor{dialinecolor}
\draw (0.8\du,-1.2\du)--(-3\du,0\du)--(0.8\du,1.2\du);}
\node[anchor=west] at (8\du,0\du){${\rm G}_2$};
\end{tikzpicture} 
\end{array}
$$
\renewcommand\arraycolsep{1.5pt}
  
Given a parabolic subgroup $P\subset G$, its intersection with $G'$ is a parabolic subgroup $P'$, which can be seen graphically by means of the folding map $p$: if $P=P(D\setminus I)$ for some subset of nodes $D\setminus I$ of the Dynkin diagram $\cD$ of $G$, then $P'=P\cap G'$ is the parabolic subgroup  of $G'$ determined by the subset of nodes $D'\setminus p(I)$. In this way, we have injective  morphisms of varieties:
$$
\cD'(p(I))\hookrightarrow \cD(I), 
$$
In particular, whenever we have two disjoint sets $I,J$ projecting onto the same subset $p(I)=p(J)=p(I\cup J)$, we will have an embedding $\sigma:\cD'(p(I))\hookrightarrow \cD(I\cup J)$. 
In the case in which $G'P(D\setminus I)=G$, that is, in the case in which $G'$ acts transitively on $\cD(I)$, then we will have $\cD'(p(I))=\cD(I)$, and the map $\sigma$ will be a nesting of type $(\cD,I,J)$. 
As noted in Remark \ref{rem:special} this happens only in the cases $\cD(I)=\DA_{2n-1}(1),$ $\DD_n(n),$ $\DB_3(1)$. In other words, only the foldings of types (\ref{sub:gpa}), (\ref{sub:gpd}) and (\ref{sub:gpb3}) provide nestings of rational homogeneous varieties.
Summing up, we have shown the following.

\begin{proposition}\label{prop:foldnest}
The foldings of the Dynkin diagrams $\cD=\DA_{2n-1}$ ($n\geq 2$), $\DD_n$ ($n\geq 4$), $\DB_3$, induce nestings of types:
$$
(\DA_{2n-1},1,2n-1),\quad (\DD_{n},n-1,n),\quad (\DB_{3},1,3).
$$
\end{proposition}

\begin{remark}
With the exception of (\ref{sub:gpb3}), all the foldings described above are induced by {\em outer automorphisms} (cf. \cite[Appendix D.~40]{FH}) of the corresponding groups $G$, in the sense that in each case there exists an outer automorphism $\phi$ of $G$ such that the subgroup $G'$ of the pair $(G',G)$ can be written as 
$$G'=\{g\in G|\,\,\phi(g)=g\}.
$$
\end{remark}

\subsection{Nestings of rational homogeneous bundles}\label{ssec:famnest}

In this section we will assume that $\Nest(\cD,I,J)\neq\emptyset$, and introduce a relative notion of nesting: 
\begin{definition}
Let $X$ be a complex manifold, 
and let $\cE\to X$ be a  
principal $G$-bundle determined by a cocycle $\theta\in \HH^1(X,G)$. Let $Y(I):=\cE\times^G\cD(I)$, $Y(I\cup J):=\cE\times^G\cD(I\cup J)$ be the associated $\cD(I)$ and $\cD(I\cup J)$-bundles over $X$, and denote by $\pi_{I\cup J,I}:Y(I\cup J)\to Y(I)$ the natural projection. A section of $\pi_{I\cup J,I}$ is called a {\em nesting of type $(\cE,I, J)$}.
\end{definition}

In order to study nestings of type $(\cE,I, J)$, we consider the action of $G$ on $\Nest(\cD,I,J)$, which allows to construct a fiber bundle:
$$
\cN(\cE,I, J):=\cE\times^G \Nest(\cD,I,J)\to X,
$$
whose fibers are isomorphic to  $\Nest(\cD,I,J)$. We may state the following:
\begin{lemma}\label{lem:nestfam}
Nestings of type $(\cE,I,J)$ are parametrized by the set of sections of $\cN(\cE,I, J)\to X$.
\end{lemma}

\begin{proof}
The bundles $Y(I)$ and $Y(I\cup J)$ trivialize with respect to the same open covering $\{U_i\}$ of $X$. By the universal property of $\Nest(\cD,I,J)$, the restriction of a nesting $\sigma$ of type $(\cE,I,J)$ to $U_i\times \cD(I)$ is given by a map $\sigma_i:U_i\to \Nest(\cD,I,J)$, for every $U_i$ of the covering. Since the bundle $\cN(\cE,I, J)$ is constructed upon the same principal bundle as $Y(I)$ and $Y(I\cup J)$, saying that the maps $\sigma_{|U_i\times \cD(I)}$ glue together to form a section of $\pi_{I\cup J,I}$ is equivalent to saying that the maps $\sigma_i$ glue together to give a section of $\cN(\cE,I,J)$.
\end{proof}

\begin{remark}
As in the case of nestings of homogeneous varieties, the above result allows us to consider the set of nestings of type $(\cE,I,J)$ as the set of closed points of a scheme $\Nest(\cE,I,J)$.
\end{remark}

Let us consider now the isotropy subgroup $G_\sigma$ of an element $\sigma\in \Nest(\cD,I,J)$ in  $G$, introduced in Remark \ref{rem:action}.
Together with \cite[Theorem~2.3]{Hus}, Lemma \ref{lem:nestfam} provides the following:
\begin{corollary}\label{cor:red}
Let $\sigma$ be a nesting of type $(\cD,I,J)$, and assume that $G$ acts transitively on $\Nest(\cD,I,J)\neq\emptyset$. Let $\cE\to X$ be a principal $G$-bundle over a complex manifold $X$. Then a nesting of type $(\cE,I,J)$ exists if and only if $\cE$ reduces to a $G_\sigma$-principal bundle, that is if the cocycle $\theta \in \HH^1(X,G)$ defining $\cE$ belongs to the image of the natural map from $\HH^1(X,G_\sigma)$. 
\end{corollary}


\section{Nestings on special rational homogeneous varieties}\label{sec:special}

The goal of this section is to describe completely the scheme $\Nest(\cD,I,J)$ in the cases in which $(\cD,I,J)$ is equal to $(\DA_{m},1,m)$, $(\DB_3,1,3)$, and $(\DD_n,n-1,n)$. More concretely, we will study each case separately in the following sections, in order to show the following:

\begin{theorem}
Let  $(\cD,I,J)$ be equal to $(\DA_{m},1,m)$ ($m\geq 2$), $(\DB_3,1,3)$, or $(\DD_n,n-1,n)$ ($n\geq 4$). If $m$ is even, then the scheme $\Nest(\DA_{m},1,m)$ is empty. In the remaining cases $\Nest(\cD,I,J)$ is isomorphic to a nonempty Zariski open set of a projective representation of the adjoint group $G$ of $\cD$, and $G$ acts transitively on $\Nest(\cD,I,J)$. 
\end{theorem}

\subsection{Nestings of type $\bm{(\DA_m,1,m)}$}\label{ssec:exAn}

Let us start with the case of the Dynkin diagram $\DA_{m}$, $m\geq 2$, and the contraction
$$\pi:=\pi_{1m,1}:\DA_{m}(1,m)\cong\P(T_{\P(V)})\lra \DA_{m}(1),$$
identifying $\DA_{m}(1)$ with the 
projectivization $\P(V)$ of a complex vector space $V$ of dimension $m+1$.

\begin{example}[Cf. {\cite[Example 2]{DCR}}]\label{ex:nestA}
Assume that $m$ is odd, let $\omega \in \bigwedge^2V^\vee$ be an antisymmetric form of maximal rank and define 
$$\sigma_\omega:\DA_m(1) \to \DA_m(1,m)$$
as the map associating to a point  $P$ in $\P(V)$, thought of as a codimension one subspace $V_m \subset V$, 
the flag  $(P\subset H(P))$, where $H(P)$ is the hyperplane corresponding to the one dimensional subspace $V_m^{\perp_\omega}$.
\end{example}

We will now prove that all the nestings of type $(\DA_m,1,m)$ are of this type, showing that:

\begin{proposition}\label{prop:nestAn}
The scheme $\Nest(\DA_m,1,m)$ is empty when $m$ is even, or isomorphic to the open set in $\P(\bigwedge^2V^\vee)$ of classes of antisymmetric forms of maximal rank, when $m$ is odd. In this case $\PGL(V)$ acts transitively on it.
\end{proposition}

\begin{proof}
The existence of a section $\sigma$ of $\pi$ is equivalent to the existence of a short exact sequence:
$$
0\ra \cN\lra T_{\P(V)}\stackrel{\theta}{\lra}\cO_{\P(V)}(d)\ra 0, 
$$
where $\cN$ is a vector subbundle of $T_{\P(V)}$ of rank $m-1$. 

Standard computations show that this is only possible if $m$ is odd and $d=2$, see \cite[Page~80]{OSS}. In fact, the existence of such exact sequence leads to the vanishing of the top Chern class $c_m(\Omega_{\P(V)}(d))$. On the other hand, we can compute this Chern class by using the Euler sequence, obtaining 
$$0=(-1)^m \sum_{i=0}^m {m+1 \choose i} (-d)^{m-i}.$$ 
Note that this tells us that $d\neq 0$. Moreover, 
multiplying by $-d$, we get $0=(1-d)^{m+1}-1$ and so we conclude that $d=2$ and $m$ is odd. 

Let us then write $m=2n-1$, $n\geq 2$. The above construction tells us that $\theta$ may be identified with an element in 
$$\HH^0(\P(V),\Omega_{\P(V)}(2))\cong \bigwedge^2V\subset \Hom(V^\vee,V),$$
that we denote also by $\theta$. The surjectivity of $\,\theta:T_{\P(V)}\to \cO_{\P(V)}(2)$ is then equivalent to the maximality of the rank of $\theta$ as an antisymmetric linear map $\theta:V^\vee\to V$. 


Denoting by $U\subset \P(\bigwedge^2V^\vee)$ the set of classes of maximal rank antisymmetric forms modulo homotheties, the above construction defines a family of nestings $U\times \P(V)\to U\times\P(T_{\P(V)})$, and so we have a morphism $\psi:U\to\Nest(\DA_{2n-1},1,{2n-1})$, which is surjective by our arguments above. Moreover, from our description above,  two antisymmetric linear maps provide the same nesting if and only if they are proportional,  i.e., $\psi$ is bijective. 

The variety $U$ is the unique open orbit of the standard action of  $\PGL(V)$ on $\P(\bigwedge^2V^\vee)$ and,  considering on $\Nest(\DA_{2n-1},1,{2n-1})$ the action described in Section \ref{sec:nesting}, one may easily check that the map $\psi:U\to \Nest(\DA_{2n-1},1,{2n-1})$ is equivariant. But then the action on $ \Nest(\DA_{2n-1},1,{2n-1})$ is transitive, and it follows that $U$ and $\Nest(\DA_{2n-1},1,{2n-1})$ are isomorphic. 
\end{proof}

 
\subsection{Nestings of type $\bm{(\DB_3,1,3)}$}\label{ssec:exB3}

In this section we will consider the following $\DB_3$-varieties:
$$
\xymatrix{&\DB_3(1,3)\ar[dl]_{\pi}\ar[dr]^p&\\\DB_3(1)&&\DB_3(3)}
$$
The variety $\DB_3(1)$ is a smooth $5$-dimensional quadric, and 
 $\DB_3(1,3)$ is a $\P^3$-bundle, which is the 
projectivization of the dual $\cS^\vee$ of the {\em spinor bundle} $\cS$ on $\DB_3(1)$ (see \cite{Ott} for details). The bundle $\cS$ is isomorphic to $\cS^\vee(-H)$  (cf. \cite[Thm.~2.8.ii]{Ott}), where $H$ denotes the ample generator of $\Pic(\DB_3(1))$, which is the class of a hyperplane section of the natural embedding of $\DB_3(1)$ as a quadric in $\P^6$. On the other hand, the variety $\DB_3(3)$ is a smooth $6$-dimensional quadric, appearing as the closed orbit of the action of the group $\Spin(7)$ on the 
projectivization of the dual of the ($8$-dimensional) spin representation, which is isomorphic to  $V_S:=\HH^0(\DB_3(1),\cS^\vee)=\HH^0(\DB_3(1),\cS(H))$. The projection onto $\DB_3(3)$ is a $\P^2$-bundle; more precisely, it is the universal family of planes contained in $\DB_3(1)$. 

In view of Proposition \ref{prop:foldnest}, we may construct a nesting of type $(\DB_3,1,3)$ upon a folding of $\DB_3$, that is upon the action of a proper subgroup $\DG_2\subset\PSO(7)$ on the $\DB_3$-varieties involved; we will show now that, essentially, all the nestings of  type $(\DB_3,1,3)$ are constructed in this way. 

Let us start by constructing a family of nestings of type $(\DB_3,1,3)$ by varying the $\DG_2$-structure of $\DB_3(1)$, that we will present written in language of octonions, for which we follow  \cite{O} (see also the references therein).

\begin{example}
Let $\O$ be the algebra of {\em complexified octonions}, defined as a $\C$-algebra with  generators $1,e_1,\dots,e_7$, and relations:
$$
\begin{array}{l}
e_i^2=-1,\quad e_ie_j=-e_je_i,\quad\mbox{for all }i,j,\\ e_1e_2=e_3,\,\, e_1e_4=e_5,\,\, e_1e_6=e_7,\,\, e_2e_4=-e_6,\,\, e_2e_5=e_7,\,\, e_3e_4=e_7,\,\, e_3e_5=e_6.
\end{array}
$$
The algebraic group $\DG_2$ can then be defined as the group of automorphisms of $\O$.

The conjugate of an element $x=x_0+\sum_{i=1}^7x_ie_i$ is defined as $x^*=x_0-\sum_{i=1}^7x_ie_i$, so that we have:
$$
N(x):=x^*x=\sum_{i=0}^7x_i^2\in \C.
$$ 
The product of octonions preserves $N$, so an element $a\in \O$ is invertible if and only if $N(a)\neq 0$. We consider now the following subvarieties of the projectivization $\P(\O^\vee)\cong\P^7$, which are invariant by the action of $\DG_2$:
$$
Q^6:=\{[x]\in\P(\O^\vee)|\,\, x^*x=0\}\subset \P(\O^\vee),\quad Q^5:=\{[x]\in\P(\O^\vee)|\,\, x^2=0\}\subset \P(\O^\vee).
$$
The first variety is a $6$-dimensional smooth quadric $Q^6\subset \P^7$, of equation $\sum_{i=0}^7x_i^2=0$, and one may easily check that the second is the intersection of $Q^6$ with the hyperplane $x_0=0$. The group $\DG_2$ acts transitively on $Q^5$ and $Q^6\setminus Q^5$ (cf. \cite[p.~89--91]{O}). 

The quadric $Q^5$ can be seen as a parameter space of a family of planes in itself; in fact, for every $P=[x]\in Q^5$, the set $s_1(P):=\{[y]\in Q^5|,\,\ xy=0\}$ is a plane in $Q^5$ containing $P$. The map $s_1$, seen as a section of $\pi:\DB_3(1,3)\to \DB_3(1)\cong Q^5$, is the nesting defined by the $\DG_2$-structure of $Q^5$. 

More generally, given any $A=[a]\in\P(\O^\vee)\setminus Q^6$ (so that $a\in\O$ is invertible), we may construct other sections, by setting:
$$
P=[x]\mapsto s_A(P):=\{[y]\in Q^5|,\,\ x(ya)=0\}. 
$$
Note that, by the properties of the octonian multiplication, $x(xa)=(xx)a=0a=0$, so the plane $s_A(P)$ passes through $P$.
This construction is of course equivalent to modifying the structure of $\DG_2$-variety on $\Q^5$ by pulling it back via the right-multiplication with $a$.
\end{example} 

Note that, in the above description, we have a natural linear isomorphism $\O^\vee\cong V_S$ \cite[Section~1.4]{O}.  The following statement tells us that every nesting of type $(\DB_3,1,3)$ can be constructed as above. 

\begin{proposition}\label{prop:nestB3} The scheme $\Nest(\DB_3,1,3)$ is isomorphic to the open set in $\P(V_S^\vee)$ parametrizing smooth hyperplane sections of $\DB_3(3)\subset\P(V_S)$, and $\PSO(8)$ acts transitively on it.
\end{proposition}

\begin{proof}
A nesting can be interpreted as a short exact sequence $$\shse{\cF^\vee}{\cS^\vee}{\cO_{\DB_3(1)}(\ell H)}$$
where $\ell\in\Z$ is an integer. We will prove that $\ell=1$. Note first that 
$$\HH^i(\DB_3(1), \Z)= \begin{cases} \Z\left\langle H^i \right\rangle &i \le 2,\\ 
 \Z\left\langle \dfrac{H^i}{2} \right\rangle & i\ge3, \end{cases}$$
hence we may write $$c_t(\cF^\vee)=1+d_1Ht+d_2H^2t^2+d_3 H^3t^3, \quad\mbox{with }d_1,d_2,2d_3\in\Z.$$ 
Recalling (see \cite[Remark~2.9]{Ott}) 
that the Chern polynomial of $\cS^\vee$ is:  
$$c_t(\cS^\vee)=1+2Ht+2H^2t^2+H^3 t^3,
$$ 
and using the equality $c_t(\cF^\vee)(1+\ell H t)=c_t(\cS^\vee)$, we immediately get 
$$d_3\ell=0,\quad d_2 \ell+d_3=1,\quad d_1 \ell +d_2=2,\quad d_1+\ell=2.$$
It follows that $\ell(\ell^3-2\ell^2+2\ell-1)=0$, and the only possible integral solutions of this equations are $\ell=0,1$. In the first case, we get the contradiction $\HH^0(\DB_3(1),\cS) \ne 0$ (see \cite[Thm.~2.3]{Ott}), so we conclude that $\ell$ is equal to $1$.

This implies that the composition of the section of $\DB_3(1,3)\to \DB_3(1)$ with the morphism onto $\DB_3(3)$ (given by the evaluation of global section of $\cS^\vee$) provides a morphism from the $5$-dimensional quadric $\DB_3(1)$ to the $6$-dimensional quadric $\DB_3(3)\subset\P(V_S)$, given by a base point free linear subsystem of $|\cO_{\DB_3(1)}(H)|$; this is only possible if this linear system is complete, and so the image of $\DB_3(1)$ is a smooth hyperplane section of $\DB_3(3)\subset\P(V_S)$. 

Let us denote by $U\subset\P(V_S^\vee)=\P(\HH^0(\DB_3(1),\cS(H))^\vee)$ the set of smooth hyperplane sections of $\DB_3(3)$. It provides a family of sections $U\times \DB_3(1)\to\DB_3(1,3)$, and hence we get a morphism $\psi:U\to \Nest(\DB_3,1,3)$. By our previous arguments, this map is surjective and, since one may easily check that two different elements $\P(\HH^0(\DB_3(1),\cS(1)))$ provide different nestings, injective. At this point, the proof follows as in the case $\DA_n$ (see the last paragraph of the proof of Proposition \ref{prop:nestAn}), from the fact that $\PSO(7)$ acts transitively on $U$. This can be proved as follows: note that $\PSO(7)$ acts on $\P(V_S^\vee)$ with an orbit isomorphic to the quadric $\DB_3(3)^\vee$ dual to $\DB_3(3)\subset\P(V_S)$; it has no fixed points, otherwise, their polar hyperplanes would be invariant, and so would be the corresponding sections of $\DB_3(3)^\vee$, a contradiction. From this it follows that the only closed orbit of the action is $\DB_3(3)^\vee$, hence, given a point $x\in U=\P(V_S^\vee)\setminus \DB_3(3)^\vee$, the closure of its orbit must contain $\DB_3(3)^\vee$, and so it must be of maximal dimension, and this may only happen if $U$ is an orbit of the action of $\PSO(7)$. 
\end{proof}


\subsection{Nestings of type $\bm{(\DD_n,n-1,n)}$}\label{ssec:exDn}

We consider now the smooth quadric $\DD_n(1)$ of dimension $(2n-2)$, which is a quotient of the group $\PSO(2n)$ of type $\DD_n$, appearing as the closed orbit of the action of this group on the 
projectivization of its natural ($2n$-dimensional) representation $V$. We denote by $q$ the quadratic form defining the quadric $\DD_n(1)\subset \P(V)$.  It is well known that we have two families of $n$-dimensional $q$-isotropic vector subspaces of $V$, parametrized by the varieties $\DD_n(n-1),\DD_n(n)$.

The variety $\DD_n(n-1,n)$ parametrizes partial flags $V_{n-1}\subsetneq V_{n}
\subsetneq V_{n+1}$ ($\dim V_i=i$) satisfying $V_{n-1}^\perp=V_{n+1}$, $V_{n}^\perp=V_{n}\in\DD_n(n-1)$, so a nesting associates to an isotropic space $V_n\subset V$ an isotropic partial flag $V_{n-1}\subsetneq V_{n}
\subsetneq V_{n+1}$.

\begin{example}\label{ex:nestD}
Let $v\in V$ be a non isotropic vector, and define
 $$\sigma_v:\DD_n(n-1)\to\DD_n(n-1,n),$$
as the map associating to $V_n\in\DD_n(n-1)$ the flag $(V_{n}+v)^\perp\subsetneq V_{n}
\subsetneq V_{n}+v$. 
\end{example}

\begin{remark}
A  non isotropic vector $v$ corresponds to a  hyperplane $H_v \subset \P(V)$ non tangent to the quadric; considering $\DD_n(n-1,n)$ as the variety parametrizing linear spaces of dimension $n-2$ contained in $\DD_n(1)$ we can describe $\sigma_v$ as the map which associates to a linear space of dimension $n-1$, parametrized by an element in $\DD_n(n-1)$, its intersection with $H_v$.
\end{remark}

We will now show that the maps $\sigma_v$ defined in Example \ref{ex:nestD} are algebraic.
Denoting by $\cQ$ the universal quotient bundle on $\DD_n(n-1)$, whose projectivization is $\DD_n(1,n-1)$, and whose space of global sections is $V$, the variety $\DD_n(n-1,n)$ is isomorphic to the projectivization of $\cQ^\vee$. 
The morphism $\pi_{n-1\,n,n}:\DD_n(n-1,n)\to \DD_n(n)$ is given by the evaluation of global sections of the first globally generated twist of $\cQ^\vee$, which is $\cQ^\vee(H)$, where  $H$ is  the ample generator of the Picard group of $\DD_n(n-1)$.

\begin{lemma}\label{lem:GstQ}
Let $v\in V$ be a non isotropic vector. Then the map $\sigma_v$ is algebraic, hence it defines a nesting of type $(\DD_n,n-1,n)$.
\end{lemma}

\begin{proof}
Saying that $v\in V=\HH^0(\DD_n(1),\cQ)$ is non isotropic is equivalent to saying that the corresponding global section  $s_v:\cO_{\DD_n(n-1)}\to \cQ$ is everywhere injective. Dualizing and tensoring with $\cO_{\DD_n(n-1)}(H)$,  we obtain a surjective morphism
$$
s_v^t:{\cQ^\vee(H)}\to{\cO_{\DD_n(n-1)}(H)},
$$
whose projectivization is the section  $\sigma_v:\DD_n(n-1)\to\DD_n(n-1,n)$.
\end{proof}

We will now prove that all the nestings of type $(\DD_n,n-1,n)$ are constructed as above. Along the proof we will need to consider the restriction of the universal quotient bundle $\cQ$ on $\DD_n(n-1)$ to the linear spaces of dimension $n-1$ contained in $\DD_n(n-1)$, which is described in the following: 

\begin{lemma}\label{lem:restrictionQ}
Let $\Lambda \subset \DD_n(n-1)$ be a linear space of dimension $n-1$. Then
 $$\cQ_{|\Lambda}\cong T_{\Lambda}(-1)\oplus \cO_{\Lambda}(1).$$
\end{lemma}

\begin{proof} The statement follows from \cite[proof of Proposition 4.5]{MOS2}), recalling that $\cQ$ is a nef and uniform bundle on $\DD_n(n-1)$, with $c_1(\cQ)=2$, since it is the restriction of the universal bundle $\cQ$ on $\G(n-1,\P(V))$ via an embedding given by $\cO_{\DD_{n}(n-1)}(2)$ (see \cite[page 17]{LVZ}).
\end{proof}

\begin{proposition}\label{prop:nestDn}
The scheme $\Nest(\DD_n,n-1,n)$ is isomorphic to the open set in $\P(V^\vee)$ parametrizing smooth hyperplane sections of the $(2n-2)$-dimensional quadric $\DD_n(1)\subset\P(V)$, and the group $\PSO(2n)$ acts transitively on it.
\end{proposition}

\begin{proof}
First of all, we will show that there exist no surjective morphisms
$$
\cQ^\vee(H)\to \cO_{\DD_n(n-1)}(kH),
$$
unless $k= 1$. Let $\cF^\vee(H)$ denote the kernel of one such surjection, and set $x:=1-k$.  Since $\cQ^\vee(H)$ is nef, we get 
$x\leq 1$; let us assume $x\neq 0$ and show how to get to a contradiction by means of cohomological computations. 

Denote by $c_i(\cF)\in \HH^{2i}(\DD_n(n-1),\Z)$, $i=0,\dots,n-1$, the Chern classes of $\cF$. Since $\cF$ is a quotient of $\cQ$, it is globally generated, and so its Chern classes are non negative (\cite[Theorem~2.5]{DPS}). Moreover, given a non isotropic vector $v \in V$, the injectivity at every point  of the corresponding global section $s_v:\cO_{\DD_n(n-1)}\to \cQ$ implies the vanishing of the top Chern class of $\cQ$:
\begin{equation}c_n(\cQ)=0.\label{eq:cnQ}\end{equation}
Hence we may write $xc_{n-1}(\cF) 
H=0$, and the effectivity of $c_{n-1}(\cF)$ together with the assumption $x\neq 0$ provides $c_{n-1}(\cF)=0$.

We now consider the restriction of the bundles $\cQ$ and $\cF$ to a projective subspace $\Lambda\simeq \P^{n-1}$; By Lemma \ref{lem:restrictionQ} we obtain the following  exact sequence:
$$
\shse{\cO_\Lambda(x)}{\cQ_{|\Lambda}=T_\Lambda(-1)\oplus\cO_{\Lambda}(1)}{\cF_{|\Lambda}}
$$
Let us write $c_i(\cF_{|\Lambda})=p_iH_{|\Lambda}^i$, where $p_0=1$, and $p_i\in \Z_{\geq 0}$ for $i=0,\dots n-1$. The Chern polynomial of  $T_\Lambda(-1)\oplus\cO_{\Lambda}(1)$ can be computed by means of the Euler sequence, and then the above exact sequence implies that:
$$
\dfrac{1-t^n}{1-t}(1+t)=(1+xt)\sum_{i=0}^{n-1}p_it^i \quad\mbox{mod }t^n,
$$
that is 
$$
(1+t)=(1+xt)(1-t)\sum_{i=0}^{n-1}p_it^i\quad\mbox{mod }t^n.
$$
From $c_{n-1}(\cF)=0$ we get $p_{n-1}=0$, and so this equation can be translated into the following equality of polynomials with integral coefficients:
$$
(1+t)=(1+(x-1)t-xt^2)\sum_{i=0}^{n-2}p_it^i+p_{n-2}xt^{n}.
$$
From this a straightforward computation provides:
$$
\left\{\begin{array}{l}
p_1+x=2\\
p_{i+1}+p_ix=p_i+p_{i-1}x,\qquad i=1,\dots,n-3\\
p_{n-2}x=p_{n-2}+xp_{n-3}
\end{array}
\right.
$$
that is:
$$
p_{n-2}x=p_{n-2}+xp_{n-3}=\ldots=p_2+p_1x=p_1+x=2.
$$
Since $p_{n-2}\geq 0$ and $x\leq 1$, $p_{n-2}x=2$ leaves us with only one possibility: $(p_{n-2},x)=(2,1)$.
But reading the above equations from right to left, $x=1$ implies $p_1=p_2=\dots=p_{n-2}=1$, a contradiction. 

We may now conclude the proof as in the case $\DA_n$ and $\DB_3$. 
Denoting by $U\subset\P(V^\vee)$ the open set parametrizing smooth hyperplane sections $H_v$ of $\DD_n(1)$, which correspond to non isotropic one dimensional subspaces $\langle v \rangle \subset V$, there exists a surjective morphism $\psi:U\to\Nest(\DD_n,n-1,n)$. On the other hand, by construction, $\psi$ is injective, and equivariant with respect to the action of $\PSO(2n)$. Hence, in order to show that it is an isomorphism it is enough to note that $\PSO(2n)$ acts transitively on $U$, which is equal to the complement in $\P(V^\vee)$ of the quadric dual to $\DD_n(1)$. 
\end{proof}


\subsection{Nestings of rational homogeneous bundles over $\P^1$}

In the three examples above we have seen that a nesting of type $(\cD,i,j)=(\DA_{n},1,n)$, $(\DB_3,1,3)$ or $(\DD_n,1,n)$ corresponds to the choice of an action on $\cD(i)$ of a subgroup $G_\sigma\subset G$ of type $\DC_{(n+1)/2}$, $\DG_2$ or $\DB_{n-1}$, respectively. As usual we denote by $G$ the adjoint group of the Dynkin diagram $\cD$.

Let us consider here the case of a complex manifold $X$ and a principal $G$-bundle $\cE\to X$, defined by a cocycle $\theta\in \HH^1(X,G)$, with $G$ the adjoint group of type $\cD=\DA_{n}$, $\DB_3$, or $\DD_n$. Given two indices $(i,j)$ so that $(\cD,i,j)$ is one of the triples described above, we may construct the bundles $\cE\times^G\cD(i)$, and $\cE\times^G\cD(i,j)$, and ask ourselves whether the natural map $\pi:\cE\times^G\cD(i,j)\to\cE\times^G\cD(i)$ admits a section, that is, in the language of Section \ref{ssec:famnest}, if we have a nesting of type $(\cE,i,j)$. 

The important point to note here is that in each one of the considered cases, the adjoint group $G$ acts transitively on $\Nest(\cD,i,j)$. Thus, by Corollary \ref{cor:red}, the existence of a nesting of type $(\cE,i,j)$ is equivalent to saying that $\cE$ reduces to a principal $G_\sigma$-bundle. 

Later on we will apply this to the case in which $X=\P^1$, in which we may use the characterization of principal bundles presented in Section \ref{ssec:groth}. As in Section \ref{ssec:folds} up to the choice of an appropriate conjugation we may assume that $H_\sigma:=H\cap G_\sigma$ is a maximal torus in $G_\sigma$, and that the induced map $\Mo(H)\to \Mo(H_\sigma)$ is given by the corresponding folding map. This allows us to write explicitly the inclusion of lattices $\Na(H_\sigma)\subset \Na(H)$,  establishing conditions on the tag $(d_1,\dots,d_n)$ of the bundle for the existence of a reduction to $G_\sigma$. 
   
A case by case straightforward analysis of the inclusion $\Na(H_\sigma)\subset \Na(H)$ provides the following:

\begin{proposition}\label{prop:nestP1}
Let $G$ be a semisimple group with Dynkin diagram $\cD$, and $i,j$ be two nodes of $\cD$, so that $(\cD,i,j)$ is, either $(\DA_n,1,n)$ with $n\geq 3$ odd, or $(\DB_n,1,n)$ with $n=3$, or $(\DD_n,n-1,n)$. Let $\cE$ be a $G$-principal bundle over $\P^1$, with tag $(d_1,\dots,d_n)$. There exists a nesting of type $(\cE,i,j)$ if and only if: 
\begin{itemize}
\item $d_i=d_{n-i}$, for every $i\in\{1,\dots,n\}$, in the case $\DA_n$,
\item $d_1=d_3$ in the case $\DB_3$, and
\item $d_{n-1}=d_n$ in the case $\DD_n$.
\end{itemize}
\end{proposition}



\section{Reductions}\label{sec:reduc}

The main goal of this section is to reduce the proof of Theorem \ref{thm:main} to the case in which $(I,J)=(\{i\},\{j\})$, and $i$ is an extremal node of $\cD$. We will divide the reduction in several steps (Propositions \ref{prop:red1}, \ref{prop:red2}, and \ref{prop:red3}), showing first how to reduce us to the case in which $J=\{j\}$. Let us start by discarding two possible nestings for the diagram $\DD_4$, that will appear in the subsequent steps of the reduction.

\begin{lemma}\label{lem:triality}
There exist no nestings of types $(\DD_4,3,\{1,4\})$, $(\DD_4,\{3,4\},1)$.
\end{lemma}

\begin{proof}
Note that the second part reduces to the first, since if we had a section $\DD_4(3,4)\to \DD_4(1,3,4)$, composing it with any section $\DD_4(3)\to \DD_4(3,4)$ given by Proposition \ref{prop:nestDn}, we would also have a section $\DD_4(3)\to \DD_4(1,3,4)$.

As usual, we consider $\DD_4(3)$  and $\DD_4(4)$ as the two parameter spaces of $\P^3$'s --say of type $a$ and $b$, respectively-- on the $6$-dimensional quadric $\DD_4(1)$. 

Assume that there exists a section $\sigma:\DD_4(3)\to\DD_4(1,3,4)$, and denote by $\sigma':\DD_4(3)\to\DD_4(3,4)$ its composition with the natural projection to $\DD_4(3,4)$. By Proposition \ref{prop:nestDn}, $\sigma'$ is determined by the choice of a $5$-dimensional smooth quadric $Q^5\subset\DD_4(1)$: given $\P^3_a\in \DD_4(3)$, its image $\sigma'(\P^3_a)$ consists of the pair $(\P^3_a,\P^3_b)$, where $\P^3_b$ is the only $\P^3_b\in\DD_4(4)$ such that $\P^3_a\cap \P^3_b=\P^3_a\cap Q^5$. 

Hence the section $\sigma$ maps $\P^3_a$ to a triple $(P,\P^3_a,\P^3_b)$, where $(\P^3_a,\P^3_b)=\sigma'(\P^3_a)$ and $P\in\P^3_a\cap\P^3_b$. Since $\P^3_a\cap\P^3_b\subset Q^5$, we may claim that the composition of $\sigma$ with the natural projection to the $6$-dimensional quadric $\DD_4(1)$ sends $\DD_4(3)$ into the quadric $Q^5$. But $\DD_4(3)$ is isomorphic to a $6$-dimensional smooth quadric, therefore this map is necessarily constant. This means that there is a point $P\in\DD_4(1)$ contained in all the $\P^3_a$'s of the family $\DD_4(3)$, a contradiction. 
\end{proof}

\begin{lemma}\label{lem:red1}
Let $G$ be a semisimple algebraic group, with associated Dynkin diagram $\cD$, and let $I, J$ be two disjoint nonempty sets of nodes of $\cD$ such that $(\cD,I,J)$ admits a nesting. Then, for every $j\in J$, $(\cD,I,j)$ admits a nesting. 
\end{lemma}

\begin{proof}
It is enough to note that the composition of a section $\cD(I)\to\cD(I\cup J)$ with the natural projection $\cD(I\cup J)\to\cD(I\cup\{j\})$ is a section of the natural projection $\cD(I\cup\{j\})\to\cD(I)$.
\end{proof}

\begin{proposition}\label{prop:red1}
It is enough to prove Theorem \ref{thm:main} in the case in which $J$ consists of only one element.
\end{proposition}

\begin{proof}
Assume that $J$ contains at least two elements, $j_1,j_2$; the previous lemma tells us that both $(\cD,I,j_1)$, $(\cD,I,j_2)$ admit a nesting. If  Theorem \ref{thm:main} holds in the case in which $J=\{j\}$, then, by inspection of the possible cases, $I$ must consist of only one extremal node of $\cD$, and both $j_1$ and $j_2$ are extremal; in particular, $\cD$ is of type $\DD_n$. But since the two triples $(\cD,i,j_1)$, $(\cD,i,j_2)$ must be of type $(\DD_n,n-1,n)$, the only possibility is that $\cD=\DD_4$, and that $\{i,j_1,j_2\}=\{1,3,4\}$.  
Up to an automorphism of $\DD_4$, we may assume $(i,j_1,j_2)=(3,1,4)$, and conclude by Lemma \ref{lem:triality}.
\end{proof}

Next step will be to reduce us to the case in which also $I$ consists of a unique element. We start with the following observation: 

\begin{remark}\label{rem:red2}
Let $G$ be a semisimple algebraic group, whose associated Dynkin diagram $\cD$ is connected. Given $I\subset D$, $j\in D$, let $i_1\in I$ be an element such that there exists a connected subdiagram of $\cD$ containing $i_1$ and $j$, and disjoint of $I\setminus \{i_1\}$; in other words, we want $i_1\in I$ to be a node neighboring the connected component containing $j$ of the Dynkin subdiagram of $\cD$ supported on the nodes $D\setminus I$. The fibers of the contractions:
$$
\pi_1:=\pi_{I,I\setminus\{i_1\}}:\cD(I)\to \cD(I\setminus\{i_1\}),\quad\pi_2:=\pi_{I\cup\{j\},I\setminus\{i_1\}}:\cD(I\cup\{j\})\to \cD(I\setminus\{i_1\}),
$$
 are rational homogeneous spaces of the form $\cD'(i_1)$ and $\cD(i_1,j)$, respectively, where $\cD'$ denotes the connected component containing $j$ of the Dynkin subdiagram of $\cD$ supported on the nodes of $D\setminus (I\setminus \{i_1\})$. If $(\cD,I,j)$ admits a nesting, then restricting it to the fibers of the above contractions we get nestings of type $(\cD',i_1,j)$.
\end{remark}

\begin{proposition}\label{prop:red2}
It is enough to prove Theorem \ref{thm:main} in the case in which $I$ and $J$ consist of only one element.
\end{proposition}

\begin{proof}
By Corollary \ref{prop:red1}, in order to prove Theorem \ref{thm:main} we may assume, without loss of generality, that $J=\{j\}$. 
Let us then assume that we have a nesting of type $(\cD,I,j)$, that is a section $\sigma:\cD(I)\to\cD(I\cup j)$. As in the previous Remark, we choose an element $i_1\in I$ such that there exists a connected subdiagram of $\cD$ containing $i_1$ and $j$, and think of $\cD(I)$, $\cD(I\cup \{j\})$ as bundles over $X:=\cD(I\setminus \{i_1\})$, with fibers $\cD'(i_1)$, $\cD'(i_1,j)$; the existence of a nesting of type $(\cD,I,J)$ implies that we have a nesting of $(\cD',i_1,j)$, over every point of $X$.

Assuming that the main theorem holds in the case in which the two defining sets of nodes consist of one element,  we get that $(\cD',i_1,j)$ is of type $(\DA_{2n-1},1,2n-1)$, $(\DB_3,1,3)$, or $(\DD_n,n-1,n)$. 

If $I$ contains at least two elements, then $\cD'$ is properly contained in $\cD$, and we may find a node $i_2\in I\setminus\{i_1\}$ neighboring $\cD'$.  

We consider now the inverse images in $\cD(I)$ and $\cD(I\cup\{j\})$ of a rational curve $C\subset X$ of type $\Gamma_{i_2}$, which are two bundles over $C$ determined by the same principal bundle, namely the one determined by the Dynkin diagram $\cD'$ tagged with the integers $d_r:=\Gamma_{i_2}\cdot K_r$ at the $r$-th node of $\cD'$ (by Proposition \ref{prop:tag}). The existence of a section $\sigma_{|C}:\cD(I)_{|C}\to\cD(I\cup\{j\})_{|C}$ tells us, by Proposition \ref{prop:nestP1}, that the integers $d_r$ satisfy a certain symmetry condition. On the other hand, we know that for the $r$-th node of $\cD'$, $d_r$ is equal to $0$ if $i_2$ is not linked to $r$, hence 
the tag contains a unique element different from zero, and 
the conditions of Proposition \ref{prop:nestP1} may be fulfilled only if, up to automorphism of $\cD$, we have:
$$\cD=\DD_n,\quad\mbox{and} \quad(i_1,j,i_2)=
\left\{\begin{array}{l}
(n-3,n-1,n)\,\,\,(\mbox{with }n=4,\mbox{ or }i_3:=n-4\in I),\\
(n,n-1,n-r)\,\,\,(\mbox{with }r\geq 3), \\
(n-1,n-3,n)\,\,\,(\mbox{with }n=4,\mbox{ or }i_3:=n-4\in I).
\end{array}\right.$$ 
We represent here the three possible cases:
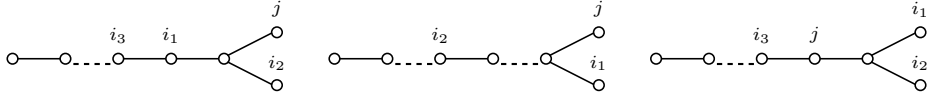
\begin{figure}[h!]
\ifx\du\undefined
  \newlength{\du}
\fi
\setlength{\du}{2\unitlength}
\begin{tikzpicture}
\pgftransformxscale{1.000000}
\pgftransformyscale{1.000000}
\definecolor{dialinecolor}{rgb}{0.000000, 0.000000, 0.000000} 
\pgfsetstrokecolor{dialinecolor}
\definecolor{dialinecolor}{rgb}{0.000000, 0.000000, 0.000000} 
\pgfsetfillcolor{dialinecolor}
\pgfsetlinewidth{0.300000\du}
\pgfsetdash{}{0pt}
\pgfsetdash{}{0pt}
\pgfpathellipse{\pgfpoint{-16\du}{0\du}}{\pgfpoint{1\du}{0\du}}{\pgfpoint{0\du}{1\du}}
\pgfusepath{stroke}
\node at (-16\du,0\du){};
\pgfpathellipse{\pgfpoint{-6\du}{0\du}}{\pgfpoint{1\du}{0\du}}{\pgfpoint{0\du}{1\du}}
\pgfusepath{stroke}
\node at (-6\du,0\du){};
\pgfpathellipse{\pgfpoint{4\du}{0\du}}{\pgfpoint{1\du}{0\du}}{\pgfpoint{0\du}{1\du}}
\pgfusepath{stroke}
\node at (4\du,0\du){};
\pgfpathellipse{\pgfpoint{14\du}{0\du}}{\pgfpoint{1\du}{0\du}}{\pgfpoint{0\du}{1\du}}
\pgfusepath{stroke}
\node at (14\du,0\du){};
\pgfpathellipse{\pgfpoint{24\du}{0\du}}{\pgfpoint{1\du}{0\du}}{\pgfpoint{0\du}{1\du}}
\pgfusepath{stroke}
\node at (24\du,0\du){};
\pgfpathellipse{\pgfpoint{34\du}{5\du}}{\pgfpoint{1\du}{0\du}}{\pgfpoint{0\du}{1\du}}
\pgfusepath{stroke}
\node at (34\du,5\du){};
\pgfpathellipse{\pgfpoint{34\du}{-5\du}}{\pgfpoint{1\du}{0\du}}{\pgfpoint{0\du}{1\du}}
\pgfusepath{stroke}
\node at (34\du,-5\du){};
\pgfsetlinewidth{0.300000\du}
\pgfsetdash{}{0pt}
\pgfsetdash{}{0pt}
\pgfsetbuttcap
{\draw (-15\du,0\du)--(-7\du,0\du);}
{\draw (5\du,0\du)--(13\du,0\du);}
{\draw (15\du,0\du)--(23\du,0\du);}
{\draw (24.65\du,0.6\du)--(33.1\du,4.9\du);}
{\draw (24.65\du,-0.6\du)--(33.1\du,-4.9\du);}
\pgfsetlinewidth{0.400000\du}
\pgfsetdash{{1.000000\du}{1.000000\du}}{0\du}
\pgfsetdash{{1.000000\du}{1.000000\du}}{0\du}
\pgfsetbuttcap
{\draw (-4.5\du,-1\du)--(2.8\du,-1\du);}
\node[anchor=south] at (4\du,1.1\du){$\scriptstyle i_3$};
\node[anchor=south] at (14\du,1.1\du){$\scriptstyle i_1$};
\node[anchor=south] at (34\du,6.1\du){$\scriptstyle j$};
\node[anchor=south] at (34\du,-4.1\du){$\scriptstyle i_2$};
\end{tikzpicture} 
\,
\ifx\du\undefined
  \newlength{\du}
\fi
\setlength{\du}{2\unitlength}
\begin{tikzpicture}
\pgftransformxscale{1.000000}
\pgftransformyscale{1.000000}
\definecolor{dialinecolor}{rgb}{0.000000, 0.000000, 0.000000} 
\pgfsetstrokecolor{dialinecolor}
\definecolor{dialinecolor}{rgb}{0.000000, 0.000000, 0.000000} 
\pgfsetfillcolor{dialinecolor}
\pgfsetlinewidth{0.300000\du}
\pgfsetdash{}{0pt}
\pgfsetdash{}{0pt}
\pgfpathellipse{\pgfpoint{-16\du}{0\du}}{\pgfpoint{1\du}{0\du}}{\pgfpoint{0\du}{1\du}}
\pgfusepath{stroke}
\node at (-16\du,0\du){};
\pgfpathellipse{\pgfpoint{-6\du}{0\du}}{\pgfpoint{1\du}{0\du}}{\pgfpoint{0\du}{1\du}}
\pgfusepath{stroke}
\node at (-6\du,0\du){};
\pgfpathellipse{\pgfpoint{4\du}{0\du}}{\pgfpoint{1\du}{0\du}}{\pgfpoint{0\du}{1\du}}
\pgfusepath{stroke}
\node at (4\du,0\du){};
\pgfpathellipse{\pgfpoint{14\du}{0\du}}{\pgfpoint{1\du}{0\du}}{\pgfpoint{0\du}{1\du}}
\pgfusepath{stroke}
\node at (14\du,0\du){};
\pgfpathellipse{\pgfpoint{24\du}{0\du}}{\pgfpoint{1\du}{0\du}}{\pgfpoint{0\du}{1\du}}
\pgfusepath{stroke}
\node at (24\du,0\du){};
\pgfpathellipse{\pgfpoint{34\du}{5\du}}{\pgfpoint{1\du}{0\du}}{\pgfpoint{0\du}{1\du}}
\pgfusepath{stroke}
\node at (34\du,5\du){};
\pgfpathellipse{\pgfpoint{34\du}{-5\du}}{\pgfpoint{1\du}{0\du}}{\pgfpoint{0\du}{1\du}}
\pgfusepath{stroke}
\node at (34\du,-5\du){};
\pgfsetlinewidth{0.300000\du}
\pgfsetdash{}{0pt}
\pgfsetdash{}{0pt}
\pgfsetbuttcap
{\draw (-15\du,0\du)--(-7\du,0\du);}
{\draw (5\du,0\du)--(13\du,0\du);}
{\draw (24.65\du,0.6\du)--(33.1\du,4.9\du);}
{\draw (24.65\du,-0.6\du)--(33.1\du,-4.9\du);}
\pgfsetlinewidth{0.400000\du}
\pgfsetdash{{1.000000\du}{1.000000\du}}{0\du}
\pgfsetdash{{1.000000\du}{1.000000\du}}{0\du}
\pgfsetbuttcap
{\draw (15.5\du,-1\du)--(22.8\du,-1\du);}
{\draw (-4.5\du,-1\du)--(2.8\du,-1\du);}
\node[anchor=south] at (4\du,1.1\du){$\scriptstyle i_2$};
\node[anchor=south] at (34\du,6.1\du){$\scriptstyle j$};
\node[anchor=south] at (34\du,-4.1\du){$\scriptstyle i_1$};
\end{tikzpicture} 
\,
\ifx\du\undefined
  \newlength{\du}
\fi
\setlength{\du}{2\unitlength}
\begin{tikzpicture}
\pgftransformxscale{1.000000}
\pgftransformyscale{1.000000}
\definecolor{dialinecolor}{rgb}{0.000000, 0.000000, 0.000000} 
\pgfsetstrokecolor{dialinecolor}
\definecolor{dialinecolor}{rgb}{0.000000, 0.000000, 0.000000} 
\pgfsetfillcolor{dialinecolor}
\pgfsetlinewidth{0.300000\du}
\pgfsetdash{}{0pt}
\pgfsetdash{}{0pt}
\pgfpathellipse{\pgfpoint{-16\du}{0\du}}{\pgfpoint{1\du}{0\du}}{\pgfpoint{0\du}{1\du}}
\pgfusepath{stroke}
\node at (-16\du,0\du){};
\pgfpathellipse{\pgfpoint{-6\du}{0\du}}{\pgfpoint{1\du}{0\du}}{\pgfpoint{0\du}{1\du}}
\pgfusepath{stroke}
\node at (-6\du,0\du){};
\pgfpathellipse{\pgfpoint{4\du}{0\du}}{\pgfpoint{1\du}{0\du}}{\pgfpoint{0\du}{1\du}}
\pgfusepath{stroke}
\node at (4\du,0\du){};
\pgfpathellipse{\pgfpoint{14\du}{0\du}}{\pgfpoint{1\du}{0\du}}{\pgfpoint{0\du}{1\du}}
\pgfusepath{stroke}
\node at (14\du,0\du){};
\pgfpathellipse{\pgfpoint{24\du}{0\du}}{\pgfpoint{1\du}{0\du}}{\pgfpoint{0\du}{1\du}}
\pgfusepath{stroke}
\node at (24\du,0\du){};
\pgfpathellipse{\pgfpoint{34\du}{5\du}}{\pgfpoint{1\du}{0\du}}{\pgfpoint{0\du}{1\du}}
\pgfusepath{stroke}
\node at (34\du,5\du){};
\pgfpathellipse{\pgfpoint{34\du}{-5\du}}{\pgfpoint{1\du}{0\du}}{\pgfpoint{0\du}{1\du}}
\pgfusepath{stroke}
\node at (34\du,-5\du){};
\pgfsetlinewidth{0.300000\du}
\pgfsetdash{}{0pt}
\pgfsetdash{}{0pt}
\pgfsetbuttcap
{\draw (-15\du,0\du)--(-7\du,0\du);}
{\draw (5\du,0\du)--(13\du,0\du);}
{\draw (15\du,0\du)--(23\du,0\du);}
{\draw (24.65\du,0.6\du)--(33.1\du,4.9\du);}
{\draw (24.65\du,-0.6\du)--(33.1\du,-4.9\du);}
\pgfsetlinewidth{0.400000\du}
\pgfsetdash{{1.000000\du}{1.000000\du}}{0\du}
\pgfsetdash{{1.000000\du}{1.000000\du}}{0\du}
\pgfsetbuttcap
{\draw (-4.5\du,-1\du)--(2.8\du,-1\du);}
\node[anchor=south] at (4\du,1.1\du){$\scriptstyle i_3$};
\node[anchor=south] at (14\du,1.1\du){$\scriptstyle j$};
\node[anchor=south] at (34\du,6.1\du){$\scriptstyle i_1$};
\node[anchor=south] at (34\du,-4.1\du){$\scriptstyle i_2$};
\end{tikzpicture} 
\end{figure}

Note that in the first case there exists a connected subdiagram of $\cD$ containing $i_2$ and $j$; by exchanging $i_2$ and $i_1$ this case reduces to the second one. 

In the third case, if  $n>4$, then there exists a connected subdiagram of $\cD$ containing $i_3$ and $j$; then by Remark \ref{rem:red2} the section $\sigma$ provides nestings of type $(\DA_n,n-2,n-1)$, contradicting Theorem \ref{thm:main} in the case where $I$ and $J$ consist of only one element. If $n=4$, this case is included in the second up to an isomorphism of $\DD_4$. 

Finally, in the second case, exchanging again $i_1$ and $i_2$, we see that the tagged Dynkin diagram $\cD'$ can only be appropriately symmetric if $i_2=n-3$ and $n=4$, or $n>4$ and $i_3:=n-4\in I$. Summing up, we are reduced to the case in which $(i_1,j,i_2)=(n,n-1,n-3)$: 
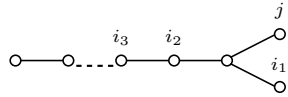
\begin{figure}[h!]
\ifx\du\undefined
  \newlength{\du}
\fi
\setlength{\du}{2\unitlength}
\begin{tikzpicture}
\pgftransformxscale{1.000000}
\pgftransformyscale{1.000000}
\definecolor{dialinecolor}{rgb}{0.000000, 0.000000, 0.000000} 
\pgfsetstrokecolor{dialinecolor}
\definecolor{dialinecolor}{rgb}{0.000000, 0.000000, 0.000000} 
\pgfsetfillcolor{dialinecolor}

\pgfsetlinewidth{0.300000\du}
\pgfsetdash{}{0pt}
\pgfsetdash{}{0pt}
\pgfpathellipse{\pgfpoint{-16\du}{0\du}}{\pgfpoint{1\du}{0\du}}{\pgfpoint{0\du}{1\du}}
\pgfusepath{stroke}
\node at (-16\du,0\du){};
\pgfpathellipse{\pgfpoint{-6\du}{0\du}}{\pgfpoint{1\du}{0\du}}{\pgfpoint{0\du}{1\du}}
\pgfusepath{stroke}
\node at (-6\du,0\du){};
\pgfpathellipse{\pgfpoint{4\du}{0\du}}{\pgfpoint{1\du}{0\du}}{\pgfpoint{0\du}{1\du}}
\pgfusepath{stroke}
\node at (4\du,0\du){};
\pgfpathellipse{\pgfpoint{14\du}{0\du}}{\pgfpoint{1\du}{0\du}}{\pgfpoint{0\du}{1\du}}
\pgfusepath{stroke}
\node at (14\du,0\du){};
\pgfpathellipse{\pgfpoint{24\du}{0\du}}{\pgfpoint{1\du}{0\du}}{\pgfpoint{0\du}{1\du}}
\pgfusepath{stroke}
\node at (24\du,0\du){};
\pgfpathellipse{\pgfpoint{34\du}{5\du}}{\pgfpoint{1\du}{0\du}}{\pgfpoint{0\du}{1\du}}
\pgfusepath{stroke}
\node at (34\du,5\du){};
\pgfpathellipse{\pgfpoint{34\du}{-5\du}}{\pgfpoint{1\du}{0\du}}{\pgfpoint{0\du}{1\du}}
\pgfusepath{stroke}
\node at (34\du,-5\du){};
\pgfsetlinewidth{0.300000\du}
\pgfsetdash{}{0pt}
\pgfsetdash{}{0pt}
\pgfsetbuttcap
{\draw (-15\du,0\du)--(-7\du,0\du);}
{\draw (5\du,0\du)--(13\du,0\du);}
{\draw (15\du,0\du)--(23\du,0\du);}
{\draw (24.65\du,0.6\du)--(33.1\du,4.9\du);}
{\draw (24.65\du,-0.6\du)--(33.1\du,-4.9\du);}
\pgfsetlinewidth{0.400000\du}
\pgfsetdash{{1.000000\du}{1.000000\du}}{0\du}
\pgfsetdash{{1.000000\du}{1.000000\du}}{0\du}
\pgfsetbuttcap
{\draw (-4.5\du,-1\du)--(2.8\du,-1\du);}
\node[anchor=south] at (4\du,1.1\du){$\scriptstyle i_3$};
\node[anchor=south] at (14\du,1.1\du){$\scriptstyle i_2$};
\node[anchor=south] at (34\du,6.1\du){$\scriptstyle j$};
\node[anchor=south] at (34\du,-4.1\du){$\scriptstyle i_1$};
\end{tikzpicture} 
\end{figure}

If $n>4$, arguing as in Remark \ref{rem:red2}, by restricting to the fibers over $\cD(I\setminus\{i_1,i_2\})$ we would obtain nestings of type $(\DD_4,\{1,3\},4)$, contradicting Lemma \ref{lem:triality}. 
\end{proof}

We finish this section by showing the following:

\begin{proposition}\label{prop:red3}
It is enough to prove Theorem \ref{thm:main} in the case in which $I=\{i\}$, $J=\{j\}$, and $i$ is an extremal node of $\cD$.
\end{proposition}

\begin{proof}
By the previous statements in this section, we may assume that $I=\{i\}$, $J=\{j\}$. Assuming that $i$ is not extremal, we consider the connected component $\ol{\cD}$ of $\cD\setminus\{i\}$ containing $j$. Moreover, we denote by $I'$ the set of nodes of $\cD$ not contained in $\ol{\cD}$, and by $\cD'$ the subdiagram of $\cD$ obtained by deleting the nodes of $I'\setminus\{i\}$. 

We then consider the following commutative diagram, where all the arrows are contractions of rational homogeneous varieties:
$$
\xymatrix{\cD(I'\cup\{j\})\ar[d]\ar[r]&\cD(i,j)\ar[d]\\\cD(I')\ar[r]&\cD(i)}
$$
By the choice of $I'$ the fibers of the two vertical maps, which are smooth morphisms, are isomorphic to $\ol{\cD}(j)$, 
hence the diagram is a Cartesian square, and a nesting of type $(\cD,i,j)$ provides a nesting of type $(\cD,I',j)$. 

Note that, since $i$ is not extremal,  $\{i\}\subsetneq I'$, and we may find a node $i_2\in I'$ neighboring $\ol{\cD}$ at the node $i$. We will consider a curve $C\subset \cD(I')$ in the class of $\Gamma_{i_2}$. Restricting the above diagram to $C$, we get a family of nestings over $C$, of type $(\cD',i,j)$. Since $i$ is extremal in $\cD'$, assuming that Theorem \ref{thm:main} holds in this case, we conclude that $(\cD',i,j)\cong(\DA_{2n-1},1,2n-1),(\DD_n,n-1,n)$, or $(\DB_3,1,3)$. But the tag of the restriction of $\cD(I'\cup\{j\})\to \cD(I')$ to $C$ is given by  
$d_r=K_r\cdot\Gamma_{i_2}$, $r\not\in I'$ (by Proposition \ref{prop:tag}), and we know (since $i_2$ is linked to $I'$ only at the node $i$) that the only index for which $d_r\neq 0$ is $r=i$. This contradicts Proposition \ref{prop:nestP1}. 
\end{proof}



\section{Proof of Theorem \ref{thm:main}}\label{sec:main}

By the results obtained in the previous section, we are left with studying, up to automorphism of the corresponding Dynkin diagrams, nestings of the following types:

\begin{subequations}
  \begin{alignat}{3}
    &(\cD,1,r),& &\qquad \cD=\DA_n,\DB_n,\DC_n,\DD_n, & &\qquad r=2,\dots,n\label{sub:first} \\
    &(\cD,n,r),& &\qquad \cD=\DB_n,\DC_n,\DD_n, & &\qquad r=1,\dots,n-1\label{sub:last}  \end{alignat}
\end{subequations}

In this section we will finish the proof of Theorem \ref{thm:main} by showing cohomological obstructions to the existence of all but the nestings listed in the statement. After describing the cohomology rings (with real coefficients) of the rational homogeneous varieties considered, we will study separately the cases (\ref{sub:first}) and (\ref{sub:last}).

\subsection{Cohomology rings of rational homogeneous varieties of classical type}\label{ssec:cohomorings}

Let us start by writing suitable presentations of the cohomology rings 
$\HH^\bullet(\cD(1)), \HH^\bullet(\cD(n)), \HH^\bullet(\cD(1,r))$ and $\HH^\bullet(\cD(r,n))$; we refer the interested reader to \cite[Section 3.1]{MOS6} for details. We may consider a set of generators of each one of these rings expressed in terms of elementary symmetric polynomials (denoted by $e_i$'s, where $i$ indicates the degree) in a set of independent variables $x_j$. We present these sets of generators in Table \ref{tab:gen}.

\begin{table}[h!]\centering\renewcommand{\arraystretch}{1.1}\small
\begin{tabular}{|c|c|rcl|l|}
\hline
Group & $\cD$ & &   &Elements &\qquad\quad Range\\
\hline
 \multirow{5}{*}{$\HH^\bullet(\cD(1))$}& all &$H$ & $=$ & $x_1$ &\\
&  $\DA$ &$A_i$  & $=$ &  $e_i(x_2, \dots, x_n)$ &\quad $i=0, \dots, n-1$\\
& $\DB, \DC$  &$K_{2i}$& $=$&  $(-1)^ie_i(x^2_{2}, \dots, x^2_{n})$ &\quad $i=0, \dots, n-1$\\
& {$\DD$}  &$K_{2i}$& $=$&  $(-1)^ie_i(x^2_{2}, \dots, x^2_{n})$ &\quad $i=0, \dots, n-2$\\
&{$\DD$}  &$\eta$&$=$& $e_{n-1}(x_2, \dots, x_n)$ \quad &\\[2 pt]
\hline\hline
$\HH^\bullet(\cD(n))$&$\DB, \DC, \DD$ &$Q_i$&  $=$ &  $e_i(x_1, \dots, x_n)$ &\quad $i=0, \dots, n$\\[2 pt]
\hline\hline
 \multirow{6}{*}{$\HH^\bullet(\cD(1,r))$} 
&all &$h$ & $=$ & $x_1$  & \\
&all &$a_i$ & $=$ &  $e_i(x_2, \dots, x_r)$ &\quad $i=0, \dots, r-1$\\
& {$\DA$} &$s_i$ & $=$ & $e_i(x_{r+1}, \dots, x_{n+1})$&\quad$i=0, \dots, n-r+1$\\
&  {$\DA$}&$q_i$ & $=$ &  $e_i(x_1, \dots, x_r)$ &\quad $i=0, \dots, r$\\
&$\DB, \DC,\DD$&$k_{2i}$ & $=$&  $(-1)^ie_i(x^2_{r+1}, \dots, x^2_{n})$ &\quad $i=0, \dots, n-r$\\
&$\DD$&$\eta_{n-r}$&$=$& $e_{n-r}(x_{r+1}, \dots, x_n)$& \quad \\[2 pt]
\hline\hline
 \multirow{2}{*}{$\HH^\bullet(\cD(r,n))$} &\multirow{2}{*}{$\DB, \DC, \DD$}&$q_i$ & $=$ &  $e_i(x_1, \dots, x_r)$ &\quad $i=0, \dots, r$\\
&&$b_i$ & $=$ & $e_i(x_{r+1}, \dots, x_{n})$&\quad$i=0, \dots, n-r$\\[2 pt]
\hline
\end{tabular}\caption{Cohomology of rational homogeneous varieties: generators.}\label{tab:gen}
\end{table}

\noindent 
In order to describe the relations among those generators, we define the following polynomials in a variable $t$:
$$ a(t) = \sum_{i=0}^{r-1}a_{i}t^{i},\quad s(t) = \sum_{i=0}^{n-r+1}s_{i}t^{i}, \quad Q(t)= \sum_{i=0}^{n}Q_{i}t^{i}.$$
$$k(t)= \sum_{i=0}^{n-r}k_{2i}t^{2i}, \quad q(t) = \sum_{i=0}^rq_{i}t^{i},\quad b(t) = \sum_{i=0}^{n-r}b_{i}t^{i}.$$ 
Then the cohomology groups $\HH^\bullet(\cD(1))$ and $\HH^\bullet(\cD(1,r))$ admit the presentations shown in Table \ref{tab:pres}, where, in each case, given a polynomial $p(t)$, with coefficients in the polynomial ring in the generators described in Table \ref{tab:gen}, $\Coeff_+(p(t))$ stands for the set of coefficients of $p$ of positive degree in the variable $t$. Note that from this description we immediately see that varieties of type $\DB_n$ and $\DC_n$ have the same cohomology. 

\begin{table}[h!]\centering\renewcommand{\arraystretch}{1.1}\small
\begin{tabular}{|c|c|c|c|}\hline
\multicolumn{2}{|c|}{Variety } & Generators & Relations\\ \hline
\multicolumn{2}{|c|}{$\DA_n(1)$}& $H$, $A_i$&$A_i -(-1)^iH^i, \,\,H^{n+1}$  \\[2pt] \hline
\multicolumn{2}{|c|}{$\DA_n(1,r)$}& $h$, $a_i$, $s_{i}$ & $\Coeff_+((1+ht)a(t)s(t))$ \\[3pt]\hline
$\DB_n(1)$ &$\DC_n(1)$& $H$, $K_{2i}$ & $K_{2i}-H^{2i},\,\, H^{2n}$\\[2pt] \hline
$\DB_n(n)$& $\DC_n(n)$& $Q_i$  &$\Coeff_+ (Q(t)Q(-t))$\\[2pt] \hline
$\DB_n(1,r)$&$\DC_n(1,r)$& $h$, $a_i$, $k_{2i}$ & $\Coeff_+ ((1-h^2t^2)a(t)a(-t) k(t))$ \\[3pt]\hline
$\DB_n(r,n)$&$\DC_n(r,n)$&  $q_i$, $b_i$ & $\Coeff_+ (q(t)q(-t) b(t)b(-t))$ \\[3pt]\hline
\multicolumn{2}{|c|}{$\DD_n(1)$}& $H$, $K_{2i}$, $\eta$ &$K_{2i}-H^{2i},  \,\,H^{2n-1},\,\, H\eta$ \\[2pt] \hline
\multicolumn{2}{|c|}{$\DD_n(n)$}& $Q_i$ &$\Coeff_+ (Q(t)Q(-t))\cup\{Q_n\}$ \\[2pt] \hline
\multicolumn{2}{|c|}{$\DD_n(1,r)$}& $h$, $a_i$, $k_{2i}$, $\eta_{n-r}$ & $\Coeff_+ ((1-h^2t^2)a(t)a(-t) k(t)),\,\, ha_{r-1}\eta_{n-r}$ \\[3pt]\hline
\multicolumn{2}{|c|}{$\DD_n(r,n)$}& $q_i$, $b_i$ & $\Coeff_+ (q(t)q(-t) b(t)b(-t)) \cup\{q_{r}b_{n-r}\}$ \\[3pt]\hline
\end{tabular}\caption{Cohomology of rational homogeneous varieties: presentations.}\label{tab:pres}
\end{table}

Let us observe that the polynomials introduced above can be thought of as Chern polynomials of certain universal vector bundles:

\begin{remark}\label{rem:chernpol} 
Let $V$ be the natural representation of a Lie algebra of type $\cD$, $\cD=\DA_n$, $\DB_n$, $\DC_n$, $\DD_n$ (which has dimension $N=n+1,2n+1,2n$, and $2n$, respectively), $\P(V)$ be its 
projectivization, and $r\leq n$ be a positive integer. 
Without loss of generality, we will assume that $\cD(r)\neq\DD_n(n-1)$. We will consider here the bundles $\cQ$, $\cS$ and $\cK$ defined in Remark \ref{rem:univbund}.
 Denoting by $p$ the natural projection from $\cD(1,r)$ or $\cD(r,n)$ to $\cD(r)$, Table \ref{tab:chern} contains an explicit expression of the Chern polynomials of $\cQ, \cS^\vee, \cK$, and  their pullbacks by $p$ (see \cite[Corollary 3.5]{MOS6}).

\begin{table}[h!]\centering
\begin{tabular}{|c|c|c|c|}\hline
Variety  &  $\cD$ &Vector bundle & Chern Polynomial\\ \hline
$\cD(n)$ &$\DB, \DC, \DD$ &$ \cQ$ & $Q(t)$\\[1 pt]
\hline
\multirow{4}{*}{$\cD(1,r)$} & all &$p^*\cQ$ & $q(t)=
(1+th)a(t)$\\
 & \DA &$p^*\cS^\vee$ & $s(t)$\\
  &$\DB, \DC, \DD$ &$p^*\cS^\vee$ & $q(-t)k(t)=
  (1-th)a(-t)k(t)$\\
 &$\DB, \DC, \DD$ &$p^*\cK$ & $k(t)$\\
\hline
\multirow{2}{*}{$\cD(r,n)$} &\multirow{2}{*}{$\DB, \DC, \DD$} &$p^*\cQ$ & $q(t)$\\
 &&$p^*\cS^\vee$ & $q(-t)b(t)b(-t)$\\[1 pt]
\hline
\end{tabular}
\caption{Chern polynomials of universal bundles.}\label{tab:chern}
\end{table}
\end{remark}

\subsection{First nodes}

In this section we will prove the following:

\begin{theorem}\label{thm:first} Let $\cD$ be a connected Dynkin diagram of classical type with $n$ nodes and $r \in \{2, \dots, n\}$. Then there are no nestings of type  $(\cD,1,r)$ unless $(\cD,1,r)$ is $(\DA_n,1,n)$ with $n$ odd, $(\DB_3,1,3), (\DD_4,1,3)$ or $(\DD_4,1,4)$.
\end{theorem}

Before proving this statement we observe that we can avoid dealing with the cases $(\DD_n,1,n)$, $(\DD_n,1,n-1)$, as follows.

\begin{remark} Assume that there exists a nesting of type $(\DD_n,1,n)$; then there exists also a nesting of type $
(\DB_{n-1},1,n-1)$. This follows from the fact that the natural projections $\DD_n(1,n)\to\DD_n(1)$, $\DB_{n-1}(1,n-1)\to\DB_{n-1}(1)$ fit into a Cartesian square:
$$
\xymatrix{\DB_{n-1}(1,n-1)\ar[d]\ar@{^{(}->}[r]&\DD_{n}(1,n)\ar[d]\\\DB_{n-1}(1)\ar@{^{(}->}[r]&\DD_{n}(1)}
$$
where the  
horizontal maps are induced by a given inclusion of the $(2n-3)$-dimensional quadric $\DB_{n-1}(1)$ in the $(2n-2)$-dimensional quadric $\DD_{n}(1)$ as a hyperplane section. Note that a $\P^{n-2}$ in $\DB_{n-1}(1)$ determines uniquely a $\P^{n-1}$ in $\DD_{n}(1)$ containing it.
The same clearly holds for a nesting of type $(\DD_n,1,n-1)$.
\end{remark}

As a first step in the proof we will show that the existence of a nesting provides an equality involving the Chern polynomials of the universal bundles introduced in Remark \ref{rem:chernpol}.

\begin{proposition}\label{prop:decomp} 
Assume that there exists a nesting of type $(\cD,1,r)$, where $\cD=\DA_n,\DB_n,\DC_n$ or $\DD_n$, different from $(\DD_n,1,n-1),(\DD_n,1,n)$, and given by a section $\sigma:\cD(1) \to \cD(1,r)$ of the natural projection $\pi:=\pi_{1r,1}:\cD(1,r) \to \cD(1)$. Consider also the contraction $p:=\pi_{1r,r}:\cD(1,r) \to \cD(r)$, and  the vector bundles $$\cQ':=(p\circ\sigma)^*\cQ,\quad  \cS':=(p\circ\sigma)^*\cS.$$
Then $\cQ'$ and $\cS'$ are two nef vector bundles satisfying $(*)$; moreover
$$P_{\cQ'}(t)P_{\cS'}(t) = 1 - (-1)^{\Cox(\cD)}t^{\Cox(\cD)},$$
where $\Cox(\cD)$ is the Coxeter number of $\cD$.\end{proposition}

\begin{proof} The nefness of $\cQ'$ and $\cS'$ follows from the fact that $\cQ$ and $\cS$ are globally generated. Recalling the definitions in Table \ref{tab:gen} and
using the properties of elementary symmetric polynomials one can show easily that:
\begin{equation} \label{eq:qketa} \pi^*A_i = \Coeff_i(a(t)s(t)),\qquad \pi^*K_{2i} = \Coeff_{2i} (a(t)a(-t)k(t)).
\end{equation}
By the commutativity of the diagram:
$$
\xymatrix@R=6pt{\HH^\bullet(\cD_n(1)) \ar@/^6mm/[rr]^{\id}\ar[r]^{\pi^*} & \HH^\bullet(\cD_n(1,r)) \ar[r]^{\sigma^*} & \HH^\bullet(\cD_n(1)) 
}$$
we have equalities:
\begin{equation} \label{eq:qketa2} A_i = \sigma^*\Coeff_i(a(t)s(t)),\qquad K_{2i} = \sigma^*\Coeff_{2i} (a(t)a(-t)k(t)),   
\end{equation}
respectively in the case of $\DA_n$, and of $\DB_n,\DC_n,\DD_n$.

From our presentations in Table \ref{tab:pres}, the cohomology groups of $\cD(1)$ are all $1$-dimensional unless $\cD=\DD_n$ and $i=n-1$,
so, to prove that $\cQ'$ and $\cS'$  satisfy $(*)$ we are left with the case $\cD=\DD_n$.
For $\cQ'$ this follows by assumption, since $\rk \cQ=r \le n-2$. By the exact sequence (\ref{eq:KSQ}), property ($*$) for $\cS'$ will follow from property ($*$) for $(p\circ \sigma)^*\cK$:  if $n$ is even we have $k_{n-1}=0$ by definition while, if $n$ is odd we can write 
$$\sigma^*k_{n-1} = xH^{n-1} + y \eta.$$ Putting together equation (\ref{eq:qketa2}) and the relation $K_{n-1}=H^{n-1}$ from Table \ref{tab:pres}, we get
 $$\sigma^*(\Coeff_{n-1} (a(t)a(-t)k(t))) = K_{n-1}= H^{n-1};$$
every summand in the left hand side except $\sigma^*k_{n-1}$ is a multiple of $H^{n-1}$, so  we get $y=0$. This gives property ($*$) for $\cS'$.\par
\medskip
Let us notice here that in the case $\cD=\DD_n$ the top Chern class of $p^*\cS$ is zero: 
$$c_{\mbox{\tiny top}}(p^*\cS)=c_{2n-r}(p^*\cS) = ha_{r-1}k_{2n-2r}=ha_{r-1}\eta_{n-r}^2=0.$$
Pulling back the equation $ha_{r-1}k_{2n-2r}=0$ via $\sigma$, since all the factors are multiples of self-intersections of $H$, we get that, either $\sigma^*a_{r-1}$, or $\sigma^*k_{2n-2r}=0$; in any case
\begin{equation}\label{eq:fuffa}
\deg \sigma^*(a(t)a(-t)k(t)) \le 2n-4.
\end{equation}
To finish the proof we will use the property ($*$) to translate the equalities
$$
\begin{array}{lcl}
\sigma^*(\Coeff_i(a(t)s(t))) = (-1)^iH^i & i=1, \dots n&\quad(\mbox{in the case of }\DA_n),\\
\sigma^*(\Coeff_{2i} (a(t)a(-t)k(t))) = H^{2i},& i=1, \dots n-1,& \quad (\mbox{in the cases of }\DB_n,\DC_n),\\
\sigma^*(\Coeff_{2i} (a(t)a(-t)k(t))) = H^{2i},\quad & i=1, \dots n-2, &\quad(\mbox{in the case of }\DD_n).
\end{array}
$$
into equalities of polynomials with rational coefficients. In fact, the polynomials $\sigma^*a(t),  \sigma^*s(t), \sigma^*k(t)$ can be written as the evaluation at $Ht$ of polynomials in one variable with coefficients in $\Q$. By abuse of notation, we will denote these polynomials by $a(t),s(t),k(t)\in\Q[t]$. In particular, we may write (see Table \ref{tab:chern}): 
$$P_{\cQ'}(t) = (1+t)a(t) \qquad 
P_{\cS'}(t) = \begin{cases} 
s(t) & \DA_n\\
(1-t)a(-t)k(t) & \DB_n, \DC_n, \DD_n \end{cases}$$
In the case of $\DA_n$,  we get  $a(t)s(t)= \sum_{i=0}^n (-1)^it^i$, therefore
$$P_{\cQ'}(t)P_{\cS'}(t)=(1+t)a(t)s(t) = 1-(-1)^{n+1}t^{n+1}.$$
\noindent In the cases of $\DB_n$ and $\DC_n$,  we get that $a(t)a(-t)k(t)= \sum_{i=0}^{n-1} t^{2i}$, hence 
$$P_{\cQ'}(t)P_{\cS'}(t)=(1+t)a(t)(1-t)a(-t)k(t) = 1-t^{2n}.$$
Finally, in the case of $\DD_n$ we get $a(t)a(-t)k(t)= \sum_{i=0}^{n-2} t^{2i}$, hence
$$P_{\cQ'}(t)P_{\cS'}(t)=(1+t)a(t)(1-t)a(-t)k(t) = 1-t^{2n-2},$$
and the Proposition is proved.
\end{proof}

\begin{proof}[Proof of Theorem \ref{thm:first}]
Assume that we have a nesting of type $(\cD,1,r)$, given by a section $\sigma:\cD(1) \to \cD(1,r)$.
By Proposition \ref{prop:decomp} we have $$P_{\cQ'}(t)P_{\cS'}(t) = 1 - (-1)^{\Cox(\cD)}t^{\Cox(\cD)}.$$
Assume first that $\Cox(\cD)$ is odd, which occurs only for $\cD=\DA_n$, with $n$ even. In this case, since $\deg(P_{\cQ'})\leq \rk(\cQ')=r$, $\deg(P_{\cS'})\leq \rk(\cS')=n-r+1$, and $n+1=\deg(P_{\cQ'})+\deg(P_{\cS'})$,  
we get that $\deg(P_{\cQ'})=\rk(\cQ')$, which is at least two, by hypothesis. By Gauss lemma, the coefficients of $P_{\cQ'}$ and  $P_{\cS'}$ are integers; since $\cQ'$ is nef, the coefficients of $P_{\cQ'}$ are nonnegative and, by Lemma \ref{lem:c1} they are all strictly positive.  Evaluating at $t=1$ we get  $P_{\cQ'}(1)P_{\cS'}(1) = 2$, which is only possible if $r=\rk(\cQ')<2$, a contradiction. 

We may then assume $\Cox(\cD)$ to be even, and apply Lemma \ref{lem:prodpol} to the bundles $\cQ'$ and $\cS'^\vee$. 
Since $\rk \cQ' >1$, either
$P_{\cS'}(t) = 1-t$, or $r=\deg P_{\cQ'} = \deg P_{\cS'} =3$. 
Since $\deg P_{\cS'}(t) =\Cox(\cD)-r$  the first case can  only happen if $\cD=\DA_n$ and $r=n$.

As for the case $r=3$, since $\Cox(\cD)=6$ we are left with the following possibilities for $(\cD,r,n)$: $(\DA_5,1,3), (\DB_3,1,3), (\DC_3,1,3), (\DD_4,1,3)$ and $(\DD_4,1,4)$.

To exclude the cases $(\DA_5,1,3)$ and $(\DC_3,1,3)$, we notice that
 the Schwarzenberger's condition $S_3^3$ (which says that $c_1(\cE)c_2(\cE)\equiv c_3(\cE)$ (mod $2$) for every vector bundle $\cE$ of rank at least three on $\P^n$, $n\geq 3$, see \cite[Section 6.1]{OSS}) excludes the existence of a vector bundle on $\P^5=\DA_5(1)=\DC_3(1)$ with Chern polynomial $1+2t+2t^2+t^3$.
\end{proof}

\subsection{Last nodes}

The proof of Theorem \ref{thm:main} will be completed by showing:

\begin{theorem}\label{thm:last} Let $\cD$ be a connected Dynkin diagram of classical type with $n$ nodes and $r \in \{1, \dots, n-1\}$. Then there are are no nestings of type  $(\cD,n,r)$ unless $(\cD,n,r)$ is $(\DA_n,n,1)$ with $n$ odd, $(\DD_n,n,n-1)$ or $(\DD_4,4,1)$.
\end{theorem}

\begin{proof}If $\cD$ is of type $\DA_n$, then the result follows from Theorem \ref{thm:first}, so we can assume that we are in cases $\DB_n, \DC_n$ or $\DD_n$.

The ring $\HH^\bullet(\cD(n))$ is generated by the Chern classes $Q_i$ of the universal quotient bundle $\cQ$, which has rank $n$, modulo the relations given by the  coefficients of positive degree terms of the polynomial $Q(t)Q(-t)$ (see Table \ref{tab:pres})
plus $Q_n$, in case $\DD_n$. These coefficients are
$$\Coeff_{2i}(Q(t)Q(-t)) =
(-1)^iQ_i^2 + \sum_{k=1}^{\min(i,n-i)}(-1)^{i-k}2Q_{i-k}Q_{i+k}
,$$
for $i=1,\dots n-1$, and $\Coeff_{2n}(Q(t)Q(-t))=(-1)^nQ_n^2$.

The relations given by the coefficients of degree $2i$, $i = 1, \dots, \lfloor{n/2}\rfloor$, read as
$$
(-1)^iQ_i^2 + \sum_{k=1}^{i-1}(-1)^{i-k}2Q_{i-k}Q_{i+k}+2Q_{2i}=0,
$$
so one may use them to write the generators of even degree of the cohomology ring in terms of the generators of odd degree. In this way we obtain presentations:
$$\begin{array}{l}\HH^\bullet(\DB_n(n))=\HH^\bullet(\DC_n(n))=  \dfrac{ \R[Q_1,Q_3,\dots,Q_{2\lfloor{(n-1)/2\rfloor}+1}]}{\big(\left\{C_{2i},\,\, i=\lfloor{n/2}\rfloor+1,\dots, n\right\}\big)},\\[10pt]
\HH^\bullet(\DD_n(n))=  \dfrac{ \R[Q_1,Q_3,\dots,Q_{2\lfloor{n/2}\rfloor-1
}]}{\big(\left\{C_{2i},\,\, i=\lfloor{(n+1)/2}\rfloor,\dots, n-1\right\}\big)},
\end{array}
$$
where each $C_{2i}$ is a homogeneous polynomial in the odd $Q_i$'s (considering each $Q_i$ as a variable of degree $i$). 

The important point to note here is that, in all cases, the minimum of the degrees of the relations is larger than the maximum of the degrees of the generators, hence these rings cannot be generated by a proper subset of the generators. 

Assume now that we have a surjective homomorphism $\sigma^*:\HH^\bullet(\cD(r,n)) \to \HH^\bullet(\cD(n))$ induced by a nesting. Since the maximum degree of the generators of $\HH^\bullet(\cD(r,n))$ is $\max(r,n-r)$, it follows that
$$\max(r,n-r) \ge  \begin{cases} 2\lfloor{(n-1)/2}\rfloor+1 &\quad\cD=  \DB_n,\DC_n,\\ 
2\lfloor{n/2}\rfloor-1 & \quad\cD=\DD_n, \end{cases}$$ 
and we are left with the following cases:

\begin{table}[h!]\centering
\begin{tabular}{|Hc|c|c|}\hline
Case & $\cD$ & $r$ & $n$\\ \hline
1&$\DB, \DC, \DD$ &$1$ &   \\
2&$ \DD$ & $n-1$& \\
3&$\DB,\DC$ & $n-1$& even\\
4&$\DD$ & $2,n-2$& odd \\
\hline
\end{tabular}
\end{table}

 The cases with $r=1$ may be discarded by observing that, in each case, the nesting would give a non constant map  $\cD(n) \to \cD(1)$, and this cannot happen since the Picard number of $ \cD(n)$ is one and $\dim \cD(1) < \dim \cD(n)$, as one may easily check. 
 
For $r=n-1$, we note first that we may dismiss the case $\cD=\DD_n$, since it has been treated in Section \ref{ssec:exDn}, and consider only the case $\cD=\DB_n, \DC_n$ ($n$ even), in which a nesting would provide a surjective map of graded algebras: 
$$\sigma^*: \dfrac{ \R[q_1,q_2,\dots, q_{n-1},b_1]}{\big(\left\{\Coeff_+ (q(t)q(-t) b(t)b(-t))\right\})}\to  \dfrac{ \R[Q_1,Q_3,\dots,Q_{n-1}]}{\big(\left\{C_{2i},\,\, i=n/2+1,\dots, n\right\}\big)}.$$ 
Since $\sigma^*$ is surjective and we do not have relations of degree smaller than $n+2$ in the target algebra, we must have 
$\sigma^*(q_{n-1}) \not \in \langle Q_1, Q_3, \dots, Q_{n-3} \rangle$. Moreover $\sigma^*(q_1)=\alpha Q_1$ with $\alpha \not = 0$, since $q_1$ is the class of an ample line bundle.  Since $n$ is even, the product $q_1q_{n-1}$ appears in the relation $\Coeff_{n}(q(t)q(-t)b(t)b(-t))$ with nonzero coefficient, hence applying $\sigma^*$ to it we get a nonzero relation of degree $n$, a contradiction. 

Finally we deal with the case of 
$(\DD_n,n,r)$ with $r=2,n-2$, and $n$ odd. 
We may assume $n \ge 5$, since $\DD_3 \simeq \DA_3$.
We start by noting that, from the description of Table \ref{tab:pres}, $\HH^\bullet(\DD_n(2,n)) \simeq \HH^\bullet(\DD_n(n-2,n))$; we may then assume $r=n-2$. Assume that we have a surjective map:
$$
\phi: \dfrac{ \R[q_1,q_2,\dots, q_{n-2},b_1,b_2]}{\big(\left\{\Coeff_+ (q(t)q(-t) b(t)b(-t))\right\} \cup \{q_{n-2}b_2\}\big)}\to  \dfrac{ \R[Q_1,Q_3,\dots,Q_{n-2}]}{\big(\left\{C_{2i},\,\, i=\frac{n+1}{2},\dots, n-1\right\}\big)}.
$$ 
Every element of degree $2$ in $\HH^\bullet(\DD_n(n))$ is a multiple of $Q_1^2$, in particular we may write $\phi(b_2)=\alpha Q_1^2$. Moreover, the surjectivity of $\phi$ implies that $\phi (q_{n-2}) \not \in \langle Q_1, Q_3, \dots, Q_{n-4} \rangle$. On the other hand, in $\HH^\bullet(\DD_n(n-2,n))$  we have the relation $q_{n-2}b_2=0$; since $n$ is odd, there are no relation of degree $n$ in $\HH^\bullet(\DD_n(n))$, therefore $\alpha=0$. In particular, $\phi$ factors through the ring  $$\frac{ \R[q_1,q_2,\dots, q_{n-2},b_1]}{\big(\left\{\Coeff_+ (q(t)q(-t) b(t)b(-t))\right\})},$$ and we get to a contradiction by following verbatim the arguments used in the case $\cD=\DB_n, \DC_n$, $n=r-1$. 
\end{proof}


\noindent{\bf Acknowledgements: }First author partially supported by MTM2015-65968-P. 
Second and third author supported by PRIN project ``Geometria delle variet\`a algebriche'', by the Miur-FFABR  2017 project, and by the Department of Mathematics of the University of Trento.
Third author partially supported by the Polish National Science Center project 2013/08/A/ST1/00804.
The results in this paper were partially obtained while the first author was a Visiting Professor at the Department of Mathematics of the University of Trento. He would like to thank this institution for its support and hospitality.
We thank the anonymous referees whose comments and suggestions were very helpful to improve the exposition of our results. 


\end{document}